\documentclass[10pt,a4paper]{article}
\usepackage[utf8]{inputenc}
\usepackage[T1]{fontenc}
\usepackage{geometry}
\usepackage{amsthm}
\usepackage[francais,english]{babel}
\usepackage{amssymb}
\usepackage{lmodern}
\usepackage{amsmath,amsfonts}
\usepackage{mathrsfs} 
\usepackage{dsfont}
\usepackage{stmaryrd}

\usepackage{graphicx}
\usepackage{fullpage}
\usepackage[nottoc,notlot,notlof]{tocbibind}
\usepackage[colorlinks=true,urlcolor=black,linkcolor=black,citecolor=black]{hyperref}
\numberwithin{equation}{section}
\usepackage{tikz}
\usetikzlibrary{matrix}
\usepackage{mathtools}
\usepackage{verbatimbox}

\begin{document}

	\title{Morse Index of Willmore spheres in $S^3$}
    \author{Alexis Michelat\footnote{Department of Mathematics, ETH Zentrum, CH-8093 Zürich, Switzerland.}}

    	\maketitle
    
    \vspace{1.5em}

    \begin{abstract}
    	We obtain an upper bound for the Morse index of Willmore spheres $\Sigma\subset S^3$ coming from an immersion of $S^2$. The quantization of Willmore energy shows that there exists an integer $m$ such that $\mathscr{W}(\Sigma)=4\pi m$. Then we show that $\mathrm{Ind}_{\mathscr{W}}(\Sigma)\leq m$. The proof relies on an explicit computation relating the second derivative of $\mathscr{W}$ for $\Sigma$ with the Jacobi operator of the minimal surface in $\mathbb{R}^3$ it is the image of by stereographic projection thanks of the fundamental classification of Robert Bryant.
    \end{abstract}

    \tableofcontents
    \vspace{0.5cm}
        \begin{center}
        \end{center}

	\theoremstyle{plain}
	\newtheorem*{theorem*}{Theorem}
	\newtheorem{theorem}{Theorem}[section]
	\newtheorem{theoremdef}{Théorème-Définition}[section]
	\newtheorem{lemme}[theorem]{Lemma}
	\newtheorem{propdef}[theorem]{Définition-Proposition}
	\newtheorem{prop}[theorem]{Proposition}
	\newtheorem{cor}[theorem]{Corollary}
	\theoremstyle{definition}
	\newtheorem*{definition}{Definition}
	\newtheorem{defi}[theorem]{Definition}
	\newtheorem{rem}[theorem]{Remark}
	\newtheorem{exemple}[theorem]{Example}
\newcommand{\N}{\ensuremath{\mathbb{N}}}
\renewcommand\hat[1]{%
\savestack{\tmpbox}{\stretchto{%
  \scaleto{%
    \scalerel*[\widthof{\ensuremath{#1}}]{\kern-.6pt\bigwedge\kern-.6pt}%
    {\rule[-\textheight/2]{1ex}{\textheight}}
  }{\textheight}%
}{0.5ex}}%
\stackon[1pt]{#1}{\tmpbox}
}
\parskip 1ex
\newcommand{\vc}[3]{\overset{#2}{\underset{#3}{#1}}}
\newcommand{\conv}[1][n]{\ensuremath{\underset{#1\rightarrow\infty}{\longrightarrow}}}
\newcommand{\A}{\ensuremath{\mathscr{A}}}
\newcommand{\D}{\ensuremath{\nabla}}
\renewcommand{\N}{\ensuremath{\mathbb{N}}}
\newcommand{\Z}{\ensuremath{\mathbb{Z}}}
\newcommand{\K}{\ensuremath{\mathscr{K}}}
\newcommand{\I}{\ensuremath{\mathbb{I}}}
\newcommand{\R}{\ensuremath{\mathbb{R}}}
\newcommand{\W}{\ensuremath{\mathscr{W}}}
\newcommand{\C}{\ensuremath{\mathbb{C}}}
\newcommand{\p}[1]{\ensuremath{\partial_{#1}}}
\newcommand{\Res}{\ensuremath{\mathrm{Res}}}
\newcommand{\lp}[2]{\ensuremath{\mathrm{L}^{#1}(#2)}}
\renewcommand{\wp}[3]{\ensuremath{\left\Vert #1\right\Vert_{\mathrm{W}^{#2}(#3)}}}
\newcommand{\np}[3][S^1]{\ensuremath{\left\Vert #2\right\Vert_{\mathrm{L}^{#3}(#1)}}}
\newcommand{\h}{\ensuremath{\mathscr{H}}}
\renewcommand{\Re}{\ensuremath{\mathrm{Re}\,}}
\renewcommand{\Im}{\ensuremath{\mathrm{Im}\,}}
\newcommand{\diam}{\ensuremath{\mathrm{diam}\,}}
\newcommand{\leb}{\ensuremath{\mathscr{L}}}
\newcommand{\supp}{\ensuremath{\mathrm{supp}\,}}
\renewcommand{\phi}{\ensuremath{\vec{\Phi}}}
\renewcommand{\psi}{\ensuremath{\vec{\Psi}}}
\renewcommand{\H}{\ensuremath{\vec{H}}}
\renewcommand{\epsilon}{\ensuremath{\varepsilon}}
\renewcommand{\bar}{\ensuremath{\overline}}
\newcommand{\s}[2]{\ensuremath{\left\langle #1,#2\right\rangle}}
\newcommand{\scal}[2]{\ensuremath{\langle #1,#2\rangle}}
\newcommand{\sg}[2]{\ensuremath{\left\langle #1,#2\right\rangle}}
\newcommand{\n}{\ensuremath{\vec{n}}}
\newcommand{\ens}[1]{\ensuremath{\left\{ #1\right\}}}
\newcommand{\lie}[2]{\ensuremath{\left[#1,#2\right]}}
\newcommand{\g}{\ensuremath{g}}
\newcommand{\e}{\ensuremath{\vec{e}}}
\newcommand{\ig}{\ensuremath{|\vec{\mathbb{I}}_{\phi}|}}
\newcommand{\ik}{\ensuremath{\left|\mathbb{I}_{\phi_k}\right|}}
\newcommand{\w}{\ensuremath{\vec{w}}}
\newcommand{\vg}{\ensuremath{\mathrm{vol}_g}}
\newcommand{\im}{\ensuremath{\mathrm{W}^{2,2}_{\iota}(\Sigma,N^n)}}
\newcommand{\imm}{\ensuremath{\mathrm{W}^{2,2}_{\iota}(\Sigma,\R^3)}}
\newcommand{\timm}[1]{\ensuremath{\mathrm{W}^{2,2}_{#1}(\Sigma,T\R^3)}}
\newcommand{\tim}[1]{\ensuremath{\mathrm{W}^{2,2}_{#1}(\Sigma,TN^n)}}
\renewcommand{\d}[1]{\ensuremath{\partial_{x_{#1}}}}
\newcommand{\dg}{\ensuremath{\mathrm{div}_{g}}}
\renewcommand{\Res}{\ensuremath{\mathrm{Res}}}
\newcommand{\res}{\mathbin{\vrule height 1.6ex depth 0pt width
			0.13ex\vrule height 0.13ex depth 0pt width 1.3ex}}
	\newcommand{\ala}[5]{\ensuremath{e^{-6\lambda}\left(e^{2\lambda_{#1}}\alpha_{#2}^{#3}-\mu\alpha_{#2}^{#1}\right)\left\langle \nabla_{\vec{e}_{#4}}\vec{w},\vec{\mathbb{I}}_{#5}\right\rangle}}
\setlength\boxtopsep{1pt}
\setlength\boxbottomsep{1pt}
\newcommand\norm[1]{%
	\setbox1\hbox{$#1$}%
	\setbox2\hbox{\addvbuffer{\usebox1}}%
	\stretchrel{\lvert}{\usebox2}\stretchrel*{\lvert}{\usebox2}%
}
\renewcommand{\qedsymbol}{}
\allowdisplaybreaks
\newcommand*\mcup{\mathbin{\mathpalette\mcapinn\relax}}
\newcommand*\mcapinn[2]{\vcenter{\hbox{$\mathsurround=0pt
			\ifx\displaystyle#1\textstyle\else#1\fi\bigcup$}}}
\newpage

\section{Introduction}
\subsection{Definitions and statement of main results}

The problem of estimating the index of a minimal surface is a rich one, with strong connections with complex analysis and algebraic geometry. In this paper the goal is to study the index of Willmore spheres $S^2\rightarrow S^3$. We first recall a few definitions. Let $(N^n,h)$ a smooth Riemannian manifold, and $\Sigma$ a (possibly non-closed) Riemann surface. For all smooth immersion $\psi:\Sigma\rightarrow N^n$, we define the Willmore functional by
\begin{align*}
	W_{N^n}(\psi)=\int_{\Sigma}|\vec{H}_g|^2\,d\vg
\end{align*}
if $g=\psi^\ast h$ is the induced metric on $\Sigma$, and $\vec{H}_g$ is the mean curvature tensor of $\psi(\Sigma)$. Then we can define a conformal Willmore functional by
\begin{align*}
	\W_{N^n}(\psi)=\int_{\Sigma}^{}\left(|\vec{H}_g|^2+K_h\right)\,d\vg
\end{align*}
if $K_h=K_{N^n}(\psi_\ast T\Sigma)$ is the sectional curvature of the $2$-plan $\psi_\ast T\Sigma\subset TN^n$.
It is conformal in the sense that if $\varphi:(N^n,h)\rightarrow (\widetilde{N}^n,\widetilde{h})$ is a conformal diffeomorphism, then for all immersion $\psi:\Sigma\rightarrow N^n$ of a closed surface $\Sigma$, we have : 
\begin{align*}
	\mathscr{W}_{N^n}(\psi)=\mathscr{W}_{\widetilde{N}^n}(\varphi\circ\psi).
\end{align*}
This is due in this general setting by Bang-Yen Chen (see \cite{chen1}, \cite{chen2}).
In the case $N^n=S^n$, we have simply
\begin{align*}
\mathscr{W}_{S^n}(\psi)=\int_{\Sigma}\left(1+|\vec{H}_g|^2\right)d\vg.
\end{align*}
 Furthermore if $p\in S^n$, and $\pi:S^n\setminus\ens{p}\rightarrow \R^n$ is a stereographic projection such that $p\notin\phi(\Sigma)\subset S^n$, as $\pi$ is conformal, we have
\begin{align*}
	\mathscr{W}_{S^n}(\psi)=\W_{\R^n}(\pi\circ \psi).
\end{align*}
In $\R^n$, we have obviously
\begin{align*}
	\W_{\R^n}(\psi)=W_{\R^n}(\psi)=\int_{\Sigma}|\vec{H}_g|^2\,d\vg
\end{align*}
However, if $\varphi:\R^n\cup\ens{\infty}\rightarrow\R^n\cup\ens{\infty}$ is a conformal transformation, we only have in general $\W(\varphi\circ \psi)=\W(\psi)$ in the case where $\varphi^{-1}(\infty)\cap\psi(\Sigma)=\varnothing$. In the special case of $\R^n$, we define a new globally conformal invariant, denoted by $\W_{\R^n}$ by abuse of notation, such that
\begin{align*}
	\W_{\R^n}(\psi)=\int_{\Sigma}(|\vec{H}_g|^2-K_g)d\vg
\end{align*}
where $K_g$ is the Gauss curvature of the surface $\psi(\Sigma)\subset\R^n$. As the $2$-form $(|\vec{H}_g|^2-K_g)d\vg$ is invariant by any conformal transformation in $\R^n$ (see \cite{blaschke}, \cite{willmore2}), \textit{a fortiori} $\W_{\R^n}$ is conformal invariant. If $\Sigma$ is closed, then by the Gauss-Bonnet theorem, we have
\begin{align*}
	\int_{\Sigma}K_gd\vg=2\pi\chi(\Sigma),
\end{align*}
where $\chi(\Sigma)$ is the Euler characteristic of $\Sigma$. So $W_{\R^n}$ and $\W_{\R^n}$ only differ by a constant in this case.
 We say that an immersion $\psi:\Sigma\rightarrow \R^3$ is a Willmore immersion if it is a critical point of $W$. We will always assume that Willmore immersions do not have branching points when they are defined from a closed Riemann surface $\Sigma$. For a smooth immersion, this is equivalent to
\begin{align*}
	\Delta_g H_g+2H_g(H_g^2-K_g)=0.
\end{align*}
if $H_g$ is the scalar curvature of $\psi(\Sigma)\subset\R^3$ (see \cite{riviere1} for a weak formulation of this formula and its consequences).
This equation goes back to the work of Blaschke (\cite{blaschke}) and Thomsen (\cite{thomsen}).  In the special case $\Sigma=S^2$, we say that a Willmore immersion is a Willmore sphere.
To define the (Morse) index, we first need a definition. 
\begin{definition}
	Let $(N^n,h)$ a Riemannian manifold that we suppose isometrically embedded in some euclidean space $\R^q$.
	We define
	\begin{align*}
		&\mathrm{W}^{2,2}(\Sigma,N^n)=\mathrm{W}^{2,2}(\Sigma,\R^q)\cap\ens{\psi:\phi(x)\in N^n\; \text{and}\; d\psi(x)\; \text{is injective for a.e.}\; x\in\Sigma}
	\end{align*}
	And for all $\psi\in\im$, we define
	\begin{align*}
		\tim{\psi}=W^{2,2}(\Sigma,TN^n)\cap\ens{\w: \w(x)\in  T_{\psi(x)}N^n\;\text{for a.e.}\; x\in\Sigma}.
	\end{align*}
\end{definition}
The index of a critical point $\psi\in\im\cap\mathrm{W}^{1,\infty}(\Sigma,N^n)$ of the Willmore functional $W$, noted $\mathrm{Ind}_W(\psi)$, is defined as the dimension of the subspace of $\tim{\psi}\cap\mathrm{W}^{1,\infty}(\Sigma,TN^n)$ where the second derivative $D^2W(\psi)$ is negative definite. Our main result is the following.
\begin{theorem}\label{main}
	Let $\psi:S^2\rightarrow S^3$ a Willmore sphere, and $m\in\N$ the integer defined by
	\begin{align*}
	m=\frac{1}{4\pi}\W(\psi)=\frac{1}{4\pi}\int_{S^2}(1+H_g^2)\,d\vg.
	\end{align*}
	Then we have
	\begin{equation}
	\mathrm{Ind}_{\W}(\psi)\leq m.
	\end{equation}
\end{theorem}

These Willmore spheres are always assumed to be defined globally, \textit{i.e.} they do not admit branching points. This implies that they are smooth (see \cite{rivierepcmi}, \cite{rivierecrelle}). This hypothesis is made in order to apply Bryant's theorem (\cite{bryant}). 

We conjecture that this bound is linearly correct, thanks of the analogy with the minimal surfaces theory (see \cite{fischer}, \cite{tysk}, \cite{ejiri1}, \cite{chodosh}). Furthermore, we think that it is possible to improve the bound $\mathrm{Ind}_W(\psi)\leq m$ (if $\W_{S^3}(\psi)=4\pi m$) to the bound $\mathrm{Ind}_W(\psi)\leq m-3$, because the first non-trivial Willmore sphere has energy $16\pi$, so for $m=4$ there should be only one direction to decrease the energy in the class of Willmore non-branched spheres. The non-existence of Willmore spheres with energy $8\pi$ and $12\pi$ is related to the non-existence of minimal surfaces with genus $0$, embedded planar ends and total curvature $-4\pi$ and $-8\pi$ (see \cite{lopez}, \cite{lopezmartin}).

\textit{Acknowledgement}. It is a pleasure to thank my advisor Tristan Rivière for useful discussions and motivation about this problem, which can be seen as a possible step towards the explicit computation of the sphere eversion (see open problems 4 and 7 in \cite{coutsphere}).

\subsection{Organisation of the paper}

We will first in section \ref{sectionW} derive the first and second variations of the Willmore functional in a general setting, and in section \ref{sectionK} the first and second variations of the Gauss curvature. We describe shortly the content of section \ref{proofmain}, dedicated to the formulas for the index and the proof of main theorem \ref{main}.

By conformity of $\W$,  the index of a Willmore sphere ${\psi}:S^2\rightarrow S^3$ is equal to the index of $\pi\circ {\psi}:S^2\rightarrow \R^3$, where $\pi$ is a stereographic projection  whose domain includes the image of ${\psi}$. 

Therefore, we fix some Willmore sphere $\psi:S^2\rightarrow\R^3$. Up to translation, by the theorem of Robert Bryant (\cite{bryant}), we have $\phi=i\circ\psi$ is a branched minimal surface with finite total curvature and planar ends (we refer to section \ref{proofmain} for the definitions), where $i:\R^3\rightarrow\R^3$ is the inversion at $0$. Therefore, we deduce that if $\vec{v}$ is a normal variation of $\psi$, then
\begin{align*}
D^2W(\psi)[\vec{v},\vec{v}]=D^2\W(\psi)[\vec{v},\vec{v}]=D^2\W(\phi)[\vec{w},\w]
\end{align*}
where $\w=|\phi|^2\vec{v}$. So the index of $\psi$ is equal to the index of $\phi$ for variations of the form $\w=|\phi|^2\vec{v}$. The latest can be explicitly computed, as for a minimal surface the index quadratic form simplifies significantly. Indeed, we have the following theorem, which is the purpose of section \ref{sectionK}.
\begin{theorem}
	Let $\Sigma$ a Riemann surface, and $\phi:\Sigma\rightarrow\R^3$ a smooth minimal immersion. Then for all normal variation $\w=w\n$, we have
	\begin{equation}\label{magiqueintro}
	D^2\W(\phi)[\w,\w]=\int_{\Sigma}^{}\left\{\frac{1}{2}(\Delta_gw-2K_gw)^2d\vg-d\left((\Delta_gw+2K_gw)\star dw-\frac{1}{2}\star d|dw|_g^2\right)\right\}
	\end{equation}
\end{theorem}
So the index of $\psi$ for $W$ (or $\W$) is equal to the index of $\phi$ for $\W$ and the special class of variations $\w=|\phi|^2v\n$. The residue term coming from the exact form will actually give all the negative directions, and it can be computed explicitly thanks of the Weierstrass-Enneper parametrisation and the planarity of the ends of the minimal surface $\phi(\Sigma)\subset \R^3$. Indeed, we have the following result.

\begin{theorem}
	Let $\Sigma$ a closed Riemann surface and $\phi:\Sigma\setminus\ens{p_1,\cdots,  p_m}\rightarrow \R^3$ a minimal immersion with $m$ embedded planar ends, such that $\psi=i\circ \phi:\Sigma\rightarrow \R^3$ is a non-branched Willmore immersion. Then for all normal variation $\vec{v}=v\n$ of $\psi$, we have
	\begin{equation}\label{magique}
	D^2\mathscr{W}(\psi)[\vec{v},\vec{v}]=\lim_{R\rightarrow 0}\frac{1}{2}\int_{\Sigma_R}(\Delta_g w-2K_gw)^2d\vg-4\pi\sum_{j=1}^m\frac{\mathrm{Res}_{p_j}\phi}{R^2}\,v^2(p_j)
	\end{equation}
	if $w=|\phi|^2v$, and $g=\phi^{\ast}g_{\,\R^3}$,
	 and
	 $
	 \displaystyle\Sigma_R=\Sigma\setminus\bigcup_{j=1}^{m}D_{\Sigma}^2(p_j,R).
	 $
\end{theorem}

The residue of the minimal immersion $\phi$ is defined as follows. By the Weierstrass-Enneper parametrisation and the planarity of the ends of $\phi$, there exists a meromorphic immersion $f:\Sigma\rightarrow \C^3$ with at most simple poles at each end $p_j\in\Sigma$ ($1\leq j\leq m$), such that $\phi=\Re f$. Therefore we can define if $f=(f_1,f_2,f_3)$, and $1\leq j\leq m$, the residue vector of $f$ as
\begin{align*}
	\mathrm{Res}_{p_j}f(z)dz=\left(\mathrm{Res}_{p_j}f_1(z)dz,\,\mathrm{Res}_{p_j}f_2(z)dz,\,\mathrm{Res}_{p_j}f_3(z)dz\right)\in\C^3\setminus\ens{0}
\end{align*}
And the quantity
\begin{align*}
	\Res_{p_j}\phi=|\Res_{p_j}f(z)dz|^2
\end{align*}
is a well-defined positive number. The quadratic form $D^2\W(\psi)$ depends only on $f$ and can be seen as a special case of a family of Schrödinger operators associated to meromorphic functions. This shows the strong analogy with the index theory for minimal surfaces (see section \ref{quadraticschro} for the general discussion, and the papers of Shiu-Yuen Cheng and Johan Tysk \cite{tysk2}, Sebasti\'{a}n Montiel and Antonio Ros \cite{montielRos} for the links between Schrödinger operators and the index of minimal surfaces).

\section{First and second variation of Willmore functional}\label{sectionW}

\subsection{Definitions and notations}

Let $(N^n,h)$ a smooth Riemannian manifold, $\D$ its Levi-Civita connection, $R$ its Riemann tensor curvature, and $M^m\subset N^n$ an isometrically embedded sub-manifold of $N^n$. We can define the induced Levi-Civita of $g$, denoted by $\bar{\D}=\iota^\ast \D$ (and $\bar{R}$ the curvature of $\bar{\D}$) and characterised by
\begin{align*}
\bar{\D} X=(\D X)^\top=\pi_M(\D X)
\end{align*}
if $X\in \Gamma(TN)$, and $\pi_M:TN\rightarrow TM$ is the orthogonal projection. As a consequence, we have if $X,Y,Z\in TM$,
\begin{align*}
\s{\D_XY}{Z}=\s{\bar{\D}_{X}Y}{Z}.
\end{align*}

Let $\vec{\I}$ the second fundamental form of $M^m\subset N^n$, the symmetric two-tensor $\vec{\I}\in\Gamma(({T^{*}M)}^{\otimes2}\otimes (TM)^\perp)$, defined for all $X,Y\in TM$, by
\begin{align*}
\vec{\I}(X,Y)=\left(\D_{X}Y\right)^\perp=\pi_M^\perp(\D_{X}Y)
\end{align*}
if $\D$ is the Levi-Civita connection of $N$, and $\pi_M^\perp:TN\rightarrow (TM)^\perp$ is the orthogonal projection. As $\D$ is torsion-free, we have
\begin{align*}
\vec{\I}(X,Y)=(\D_{X}Y)^\perp=(\D_{Y}X-[X,Y])^\perp=(\D_{Y}X)^\perp=\vec{\I}(Y,X)
\end{align*}
The main symmetries of the second fundamental form are gathered in the following theorem (see \cite{paulin}), which we explicitly recall for the convenience of the reader and to fix notations.
\begin{theorem*}
	\begin{itemize}
		\item[$\mathrm{(i)}$] (\textbf{Gauss formula}) For all $X,Y\in \Gamma(TM)$, 
		\begin{align*}
		\D_XY=\bar{\D}_XY+\vec{\I}(X,Y)
		\end{align*}
		\item[$\mathrm{(ii)}$] (\textbf{Gauss equation)} For all $X,Y,Z,W\in\Gamma(TM)$,
		\begin{align*}
		\bar{R}(X,Y,Z,W)=R(X,Y,Z,W)+\s{\vec{\I}(Y,Z)}{\vec{\I}(X,W)}-\s{\vec{\I}(X,Z)}{\vec{\I}(Y,W)}.
		\end{align*}
		\item[$\mathrm{(iii)}$] (\textbf{Codazzi-Mainardi identity}) For all $X,Y,Z\in\Gamma(TM)$, we have
		\begin{align*}
		(\D_X^\perp \vec{\I})(Y,Z)=(\D_Y^\perp \vec{\I})(X,Z)+(R(X,Y)Z)^\perp.
		\end{align*}
	\end{itemize}
\end{theorem*}

In the following, we assume $N$ to be $3$-dimensional. Let $\Sigma$ a Riemann surface,  and $\phi\in\mathrm{W}^{2,2}(\Sigma,N^3)$ a smooth immersion. We restrict ourselves in the following computations to dimension $3$ only to simplify the presentation, as we will deal with a local normal unit vector-field inducing locally the second fundamental form, whereas for a $n$-manifold, we need to deal with a $(n-2)$-vector field, adding sums only in computations, and not in final formulas. Let $g=\phi^\ast h$ the induced metric on $\Sigma$, $(g_{i,j})_{1\leq i,j\leq 2}$ its local components and $(g^{i,j})_{1\leq i,j\leq 2}$ the components of the inverse of $g$. We define the mean-curvature tensor field $\vec{H}_g$ of the immersion $\phi:\Sigma\rightarrow N^3$ by
\begin{align*}
	\vec{H}_g=\frac{1}{2}\sum_{i,j=1}^2g^{i,j}\vec{\I}_{i,j}
\end{align*}
where for $1\leq i,j\leq 2$, we define $\vec{\I}_{i,j}=\vec{\I}(\e_i,\e_j)=(\D_{\e_i}\e_j)^\perp$ if $\D$ is the Levi-Civita connection on $(N^3,h)$, $\e_k=\d{k}\phi$ for $k=1,2$, and $\perp:TN^3\rightarrow \phi_\ast(T\Sigma)^\perp$ is the orthogonal projection.

\subsection{First variation of $W$}

We define for $i=1,2$, $\e_i=\d{i}\phi$, and we use a conformal local chart where
\begin{align*}
\s{\e_i}{\e_j}=e^{2\lambda}\delta_{i,j},
\end{align*}
and we want to compute the first and second variation of $\phi\in\im\cap W^{1,\infty}(\Sigma,N^3)$. First suppose that $\phi:\Sigma\rightarrow N$ is smooth. A variation of $\phi$ if a $C^2$ function ${\phi_t}\in C^2(I,\mathrm{W}_\iota^{2,2}(\Sigma,N^3)\cap C^{\infty}(\Sigma,N^3))$  from an open interval $I$ of $\R$ containing $0$, such that $\phi_0=\phi$. By abuse of notation the variation of $\phi$ is
\begin{align*}
\w=\left(\frac{d}{dt}\phi_t\right)_{|t=0}.
\end{align*}

\subsubsection{First variation of the metric}

We first compute the variation of the metric, its inverse and the induced volume form.
\begin{lemme}\label{lemme1}
	Under the preceding hypotheses, we have for $1\leq i,j\leq 2$, for all $t\in I$,
	\begin{align}
		&\frac{d}{dt}g_{i,j}=\s{\D_{\e_i}\w}{\e_j}+\s{\D_{\e_j}\w}{\e_i},\\
		&\frac{d}{dt}g^{i,j}=-e^{-4\lambda}\left(\s{\D_{\e_i}\w}{\e_j}+\s{\D_{\e_j}\w}{\e_i}\right),\\
		&\frac{d}{dt}d\vg=\s{d\phi}{d\w}_gd\vg.
	\end{align}
\end{lemme}
\begin{proof}
	Fix $1\leq i,j\leq 2$. Locally, we have $g_{i,j}=\s{\d{i}\phi_t}{\d{j}\phi_t}$. Therefore, the compatibility with the metric gives
	\begin{align*}
		\frac{d}{dt}\s{\d{i}\phi_t}{\d{j}\phi_t}&=\s{\D_{\frac{d}{dt}}\d{i}\phi}{\d{j}\phi_t}+\s{\D_{\frac{d}{dt}}\d{j}\phi_t}{\d{i}\phi_t}\\
		&=\s{\D_{\e_i}\w}{\e_j}+\s{\D_{\e_j}\w}{\e_i}
	\end{align*}
Then,
\begin{align*}
\frac{d}{dt}\det {g_t}&=\frac{d}{dt}\left(g_{1,1}g_{2,2}-g_{1,2}^2\right)\\
&=2\left(\s{\D_{\e_1}\w}{\e_1}g_{2,2}+\s{\D_{\e_2}\w}{\e_2}g_{1,1}-(\s{\D_{\e_1}\w}{\e_2}+\s{\D_{\e_2}\w}{\e_1})g_{1,2}\right)
\end{align*}
which gives at $t=0$
\begin{align*}
{\frac{d}{dt}\det{g_t}}_{|t=0}&=2e^{2\lambda}\left(\s{\D_{\e_1}\w}{\e_1}+\s{\D_{\e_2}\w}{\e_2}\right)\\
&=2e^{4\lambda}\s{d\phi}{d\w}_g
\end{align*}
and this implies as $d\mathrm{vol}_{g_t}=\sqrt{\det g_t}\,dx_1\wedge dx_2$ locally that
\begin{align*}
{\frac{d}{dt}\mathrm{vol}_{g_t}}_{|t=0}=\s{d\phi}{d\w}_gd\vg
\end{align*}
Now, the explicit formula
 \begin{align*}
 	g^{i,j}=(-1)^{i+j}\frac{g_{i+1,j+1}}{\det g}
 \end{align*}
 gives
 \begin{align*}
 	\frac{d}{dt}g^{i,j}=(-1)^{i+j}\left(e^{-4\lambda}\left(\s{\D_{\e_{i+1}}\w}{\e_{j+1}}+\s{\D_{\e_{j+1}}\w}{\e_{i+1}}\right)-\delta_{i,j}e^{2\lambda}\frac{2e^{2\lambda}(\s{\D_{\e_i}\w}{\e_i}+\s{\D_{\e_{i+1}}\w}{\e_{i+1}})}{e^{8\lambda}}\right)\\
 \end{align*}
 so if $i=j$, we get
 \begin{align*}
 	\frac{d}{dt}g^{i,i}=-2e^{-4\lambda}\s{\D_{\e_i}\w}{\e_i}
 \end{align*}
 and
 \begin{align*}
 	\frac{d}{dt}g_{1,2}=-e^{-4\lambda}\left(\s{\D_{\e_1}\w}{\e_2}+\s{\D_{\e_2}\w}{\e_1}\right)
 \end{align*}
 so we deduce that
 \begin{align*}
 	\frac{d}{dt}g^{i,j}=-e^{-4\lambda}\left(\s{\D_{\e_i}\w}{\e_j}+\s{\D_{\e_i}\w}{\e_j}\right).
 \end{align*}
 which concludes the proof of the lemma.
\end{proof}

\subsubsection{First variation of the second fundamental form}

We note that even if we have chosen a local chart where $\phi$ is conformal, for $t\neq 0$, $\phi_t$ is not conformal in general, and as we aim as computing second derivative, we must keep track of the exact quantities depending on $t$. Therefore, we introduce the following quantities.

\begin{defi}
	For all $t\in I$, let
	\begin{align*}
		e^{2\lambda(t)}=|\d{1}\phi_t\wedge \d{2}\phi_t|,\quad e^{2\lambda_i(t)}=g_{i,i}(t),\quad\mu(t)=g_{1,2}(t)=\s{\d{1}\phi_t}{\d{2}\phi_t}
	\end{align*}
\end{defi}

\begin{lemme}\label{lemmeortho}
	Let $\vec{v}_t\in\Gamma(TN^3)$ a time-dependant smooth vector field on $N^3$, and denote $\pi_{\phi_t}$ the orthogonal projection on $TN^3\rightarrow \phi_{t\ast}(T\Sigma)$. Then we have
	\begin{align*}
		\pi_{\phi_t}(\vec{v}_t)&=g_t^{-1}
		\left(\begin{aligned}
		&\s{\vec{v}_t}{\d{1}\phi_t}\\
		&\s{\vec{v}_t}{\d{2}\phi_t}
		\end{aligned}\right)(\d{1}\phi_t,\,\d{2}\phi_t)
	\end{align*}
	where $g_t^{-1}$ is the inverse of the metric $g_t=\phi_t^\ast h$, viewed as a squared 2-matrix.
\end{lemme}

\begin{lemme}\label{lemme2}
	For all $1\leq i,j\leq 2$,
	\begin{equation}
		\D_{\frac{d}{dt}}^\perp \vec{\I}_{i,j}=\left((\D_{\e_i}\D_{\e_j}-\D_{(\D_{\e_i}\e_j)^\top})\w+R(\w,\e_i)\e_j\right)^\perp.
	\end{equation}
\end{lemme}
\begin{proof}
As $\n$ is a unit vector, we have $\D_{\frac{d}{dt}}^\perp\n_t=0$, as
\begin{align*}
\s{\D_{\frac{d}{dt}}\n_t}{\n_t}=\s{\D_{\frac{d}{dt}}^\perp\n_t}{\n_t}=0,
\end{align*}
and furthermore, by lemma \eqref{lemmeortho},
\begin{align*}
\D_{\frac{d}{dt}}\n_t&=
-e^{-4\lambda(t)}\left(e^{2\lambda_2(t)}\s{\n_t}{\D_{\d{1}\phi_t}\w_t}-\mu(t)\s{\n_t}{\D_{\d{2}\phi_t}\w_t}\right)\d{1}\phi_t\\
	&-e^{-4\lambda(t)}\left(-\mu(t)\s{\n_t}{\D_{\d{1}\phi_t}\w_t}+e^{2\lambda_1(t)}\s{\n_t}{\D_{\d{2}\phi_t}\w_t}\right)\d{2}\phi_t\\
	&=-e^{-4\lambda}\left(e^{2\lambda_2}\s{\n}{\D_{\e_1}\w}-\mu\s{\n}{\D_{\e_2}\w}\right)\e_1
	-e^{-4\lambda}\left(-\mu\s{\n}{\D_{\e_1}\w}+e^{2\lambda_1}\s{\n}{\D_{\e_2}\w}\right)\e_2
\end{align*}
where $\w_t=\D_{\frac{d}{dt}}\phi_t$, dropping the $t$ index on the last line.
And, if $\vec{v}_t$ is a smooth vector-field on $N^3$ depending of $t$, then we have
\begin{align*}
	&\s{\vec{v}_t}{\D_{\frac{d}{dt}}\n_t}=\s{\vec{v}_t}{-e^{-4\lambda}\left(e^{2\lambda_2}\s{\n}{\D_{\e_1}\w}-\mu\s{\n}{\D_{\e_2}\w}\right)\e_1
		-e^{-4\lambda}\left(-\mu\s{\n}{\D_{\e_1}\w}+e^{2\lambda_1}\s{\n}{\D_{\e_2}\w}\right)\e_2}\\
	&=-e^{-4\lambda}\left(e^{2\lambda_2}\s{\vec{v}_t}{\e_1}-\mu(t)\s{\vec{v}_t}{\e_2}\right)\s{\n}{\D_{\e_1}\w}-e^{-4\lambda}\left(e^{-2\lambda_1}\s{\vec{v}_t}{\e_2}-\mu\s{\vec{v}_t}{\e_1}\right)\s{\n}{\D_{\e_2}\w}\\
	&=-\s{\n_t}{\D_{\pi_{\phi}(\vec{v}_t)}\w_t}
	=-\s{\n_t}{\D_{(\vec{v}_t)^\top}\w_t}
\end{align*}
so we get
\begin{equation}\label{normalderivative}
\D_{\frac{d}{dt}}^\perp(\vec{v}_t)^\perp=\D_{\frac{d}{dt}}^\perp\vec{v}_t-\D_{\vec{v}_t^\top}^\perp\w_t.
\end{equation}
Therefore
\begin{align*}
\D_{\frac{d}{dt}}^\perp\vec{\I}_{i,j}=\s{\D_{\frac{d}{dt}}\D_{\e_i}\e_j}{\n}\n-\s{\n}{\D_{(\D_{\e_i}\e_j)^\top}\w}\n
\end{align*}
and
\begin{align*}
\D_{\frac{d}{dt}}\D_{\e_i}\e_j&=\D_{\e_i}\D_{\frac{d}{dt}}\e_j+R(\w,\e_i)\e_j\\
&=\D_{\e_i}\D_{\e_j}\w+R(\w,\e_i)\e_j
\end{align*}
implies that
\begin{align}
\D_{\frac{d}{dt}}^\perp\vec{\I}_{i,j}&=\s{\D_{\e_i}\D_{\e_j}\w+R(\w,\e_i)\e_j-\D_{(\D_{\e_i}\e_j)^\top}\w}{\n}\n\nonumber\\
&=\s{(\D_{\e_i}\D_{\e_j}-\D_{(\D_{\e_i}\e_j)^\top})\w+R(\w,\e_i)\e_j}{\n}\n\nonumber\\
&=\left((\D_{\e_i}\D_{\e_j}-\D_{(\D_{\e_i}\e_j)^\top})\w+R(\w,\e_i)\e_j\right)^\perp\nonumber
\end{align}
which concludes the proof of the lemma.
\end{proof}

\subsubsection{First variation of the mean curvature}
We deduce from \eqref{lemme1} and \eqref{lemme2} that, making the shift of notation $\e_i=e^{-\lambda}\d{i}\phi$ (therefore $(\e_1,\e_2)$ is an orthonormal frame for $\phi$)
\begin{align*}
\D_{\frac{d}{dt}}^\perp H_{g_t}&=\frac{1}{2}\sum_{i,j=1}^{2}\left(\frac{d}{dt}g^{i,j}\right)\vec{\I}_{i,j}+g^{i,j}\left(\D_{\frac{d}{dt}}^{\perp}\vec{\I}_{i,j}\right)\\
&=\frac{1}{2}\sum_{i,j=1}^2\bigg(-\left(\s{\D_{\e_i}\w}{\e_j}+\s{\D_{\e_j}\w}{\e_i}\right)\vec{\I}_{i,j}\bigg)\\
&+\frac{1}{2}e^{-2\lambda}\sum_{k=1}^2\left(\left(\D_{\e_k}\D_{\e_k}-\D_{(\D_{\e_k}\e_k)^\top}\right)\w+R(\w,\e_k)\e_k\right)^\perp\\
&=\frac{1}{2}\left(\Delta^{\n}_g\w+\mathscr{R}_1^{\n}(\w)-\sum_{i,j=1}^2\left(\s{\D_{\e_i}\w}{\e_j}+\s{\D_{\e_j}\w}{\e_i}\right)\vec{\I}_{i,j}\right).
\end{align*}
If $\Delta^{\n}_g$ is the normal Laplacien, defined by (see the book of Tobias Colding and William Minicozzi \cite{coldingminicozzi1})
\begin{align*}
	\Delta_g^{\n}\w=\left(\sum_{k=1}^{2}\left(\D_{\e_k}\D_{\e_k}-\D_{\left(\D_{\e_k}\e_k\right)^\top}\right)\w\right)^\perp,
\end{align*}
and 
\begin{align*}
	\mathscr{R}_1^{\n}(\w)=\left(\sum_{k=1}^2 R(\w,\e_k)\e_k\right)
\end{align*} 
Therefore, we deduce that
\begin{align*}
	&DW(\phi)\cdot \w={\left(\frac{d}{dt}\int_{\Sigma}^{}|\vec{H}_{g_t}|^2d\mathrm{vol}_{g_t}\right)}_{|t=0}=\left(\int_{\Sigma}^{}2\s{\D_{\frac{d}{dt}}\vec{H}_{g_t}}{\vec{H}_{g_t}}d\mathrm{vol}_{g_t}+\int_{\Sigma}^{}|\vec{H}_{g_t}|^2\left(\frac{d}{dt}d\mathrm{vol}_{g_t}\right)\right)_{|t=0}\\
	&=\int_{\Sigma}^{}\s{\Delta^{\n}_g\w+\mathscr{R}_1^{\n}(\w)-\sum_{i,j=1}^2\left(\s{\D_{\e_i}\w}{\e_j}+\s{\D_{\e_j}\w}{\e_i}\right)\vec{\I}_{i,j}}{\vec{H}_g}d\vg+\int_{\Sigma}^{}|\vec{H}_g|^2\s{d\phi}{d\w}d\vg.
\end{align*}
We note that this formula holds for the minimal regularity assumption \textit{i.e.} for $\phi\in\im\cap W^{1,\infty}(\Sigma,N^3)$. Furthermore, as mentioned in the introduction, it does not depend on the dimension of $N$, and is actually valid in any dimension. Indeed, in a Riemannian manifold $(N^n,h)$ one simply needs to replace $\n$ with a $(n-2)$-vector inducing the second fundamental form, still denoted by $\n$. Then locally, $\n=\n_1\wedge\cdots \n_{n-2}$ where $(\n_1,\cdots,\n_{n-2})$ is an orthonormal basis of the normal bundle of $\phi_t$. Extending by parallel transport the $\n_j$ ($1\leq j\leq n-2$) such that $\D_{\frac{d}{dt}}^\perp \n_j=0$, the formula \eqref{normalderivative} is still correct and we get immediately the result.

If $\phi$ is smooth, and $\w$ is a normal variation, $\s{\D_{\e_i}\w}{\e_j}=-\s{\w}{\n}\I_{i,j}=-\s{\w}{\vec{\I}_{i,j}}$, so we get
\begin{align*}
\frac{d}{dt}\int_{\Sigma}H_g^2d\vg&=\int_{\Sigma}\s{\Delta^{\n}_g\w+\mathscr{R}^\perp(\w)+2\sum_{i,j=1}^2\s{\w}{\vec{\I}_{i,j}}\vec{\I}_{i,j}}{\vec{H}_g}d\vg\\
&-2\int_{\Sigma}|\vec{H}_g|^2\s{\vec{H}_g}{\w}d\vg\\
&=\int_{\Sigma}\s{\Delta^{\n}_g\vec{H}_g-2|\vec{H}_g|^2\vec{H}_g+2\mathscr{A}(\vec{H}_g)+\mathscr{R}_1^{\n}(\vec{H}_g)}{\w}d\vg
\end{align*}
This gives the classical Willmore equation (see the paper of Joel Weiner \cite{weiner}, and note the different conventions we use here)
\begin{equation}
\Delta^{\n}_g\vec{H}_g-2|\vec{H}_g|^2\vec{H}_g+2\mathscr{A}(\vec{H}_g)+\mathscr{R}_1^{\n}(\vec{H}_g)=0.
\end{equation}
which is valid in an arbitrary Riemannian manifold $(N^n,h)$ such that $\phi:\Sigma\rightarrow N^n$ is an immersion. Here $\mathscr{A}$ is the Simons' operation, defined by
\begin{align*}
\mathscr{A}(\vec{H}_g)=\sum_{i,j=1}^2\s{\vec{H}_g}{\vec{\I}(\vec{e}_i,\vec{e}_j)}\vec{\I}(\vec{e},\vec{e}_j)
\end{align*}
if $(\vec{e}_1,\vec{e}_2)$ is a orthonormal frame on $\phi_{\ast}(T\Sigma)$ (we recall the shift of notation $\e_i=e^{-\lambda}\d{i}\phi$).
In dimension $3$, the equation can sometimes be written in a simpler way if $\phi(\Sigma)$ has a trivial normal bundle. Indeed, in this case, we can define up to the sign of a normal vector-field, the scalar mean curvature $H_g$, defined by $\vec{H}_g=H_g\n$. This gives
\begin{align*}
\Delta_g^{\n}\vec{H}_g&=\left((\Delta_gH_g)\n+H_g(\D_{\vec{e}_1}^2\n+\D_{\vec{e}_2}^2\n)\right)^\perp\\
&
=(\Delta_g H_g)\n-H_g\left(|\D_{\vec{e}_1}\n|^2+|\D_{\vec{e}_2}\n|^2\right)\n\\
&=\left((\Delta_g H_g)-|\vec{\I}_g|^2H_g\right)\n.
\end{align*}
Then
\begin{align*}
\mathscr{A}(\vec{H_g})=|\vec{\I}_g|^2H_g\n
\end{align*}
and finally,
\begin{align*}
\Delta^{\n}_g\vec{H}_g-2|\vec{H}_g|^2\vec{H}_g+2\mathscr{A}(\vec{H}_g)+\mathscr{R}^\perp(\vec{H}_g)&=\left(\Delta_gH_g-2H_g^3+|\vec{\I}_g|^2H_g+\mathrm{Ric}(\n,\n)H_g\right)\n\\
&=\left(\Delta_gH_g-2H_g^3+(4H_g^2-2K_g)H_g+\mathrm{Ric}(\n,\n)H_g\right)\n\\
&=\left(\Delta_gH_g+2H_g(H_g^2-K_g)+\mathrm{Ric}(\n,\n)H_g\right)\n
\end{align*}
which finally gives
\begin{equation}
\Delta_gH_g+2H_g(H_g^2-K_g)+\mathrm{Ric}(\n,\n)H_g=0
\end{equation}
if $\mathrm{Ric}$ is the Ricci curvature of $(N^3,h)$.

\subsection{Second variation of $W$}

Let $\phi:\Sigma\rightarrow N^3$ a smooth critical point of $W$. Then the second variation of $W$ is well-defined, and does not depend on the variation $\phi_t$ such that 
\begin{align*}
	\w=\left(\frac{d}{dt}\phi_t\right)_{|t=0}.
\end{align*}
Therefore we choose a variation $\phi_t$ such that
\begin{align*}
	\D_{\frac{d}{dt}}\frac{d}{dt}\phi_t=0.
\end{align*}
and we abbreviate this expression by abuse of notation as $\D_{\w}\w=0$.

\subsubsection{Second variation of the metric}

We split the preliminary computation into two lemmas.
\begin{lemme}\label{lemme3}
	Let $1\leq i,j\leq 2$. We have 
	\begin{align}
		\left(\frac{d^2}{dt^2}g_{i,j}\right)_{|t=0}&=2\s{\D_{\e_i}\w}{\D_{\e_j}\w}-2\s{R(\e_i,\w)\w}{\e_j},\\
		\left(\frac{d^2}{dt^2}g^{i,j}\right)_{|t=0}&=2e^{-2\lambda}\left(-\s{\D_{\e_i}\w}{\D_{\e_j}\w}_g+\s{R(\e_i,\w)\w}{\e_j}_g\right)\nonumber\\
		&+4e^{-2\lambda}\left(\s{\D_{\e_i}\w}{\e_j}_g+\s{\D_{\e_j}\w}{\e_i}_g\right)\left(\s{\D_{\e_1}\w}{\e_1}_g+\s{\D_{\e_2}\w}{\e_2}_g\right)\nonumber\\
		&-2\delta_{i,j} e^{-2\lambda}\left(4e^{-2\lambda}\s{\D_{\e_1}\w}{\e_1}\s{\D_{\e_2}\w}{\e_2}_g-\left(\s{\D_{\e_1}\w}{\e_1}_g+\s{\D_{\e_2}\w}{\e_2}_g\right)^2\right),\\
		\left(\frac{d^2}{dt^2}d \mathrm{vol}_{g_t}\right)_{|t=0}&=\bigg(|d\w|_g^2-\mathscr{R}_2(\w,\w)-2\s{\D_{\e_1}\w}{\e_1}_g^2-2\s{\D_{\e_2}\w}{\e_2}_g^2\nonumber\\
		&\quad\quad\quad\quad\quad\quad\quad\quad\quad-\left(\s{\D_{\e_1}\w}{\e_2}_g+\s{\D_{\e_2}\w}{\e_1}_g\right)^2\bigg)d\vg.
	\end{align}
	where
	\begin{align*}
		\mathscr{R}_2(\w,\w)=e^{-2\lambda}\sum_{i=1}^{2}\s{R(\e_i,\w)\w}{\e_i}.
	\end{align*}
\end{lemme}
\begin{proof}
	As $\D_{\w}\w=0$, we have
	\begin{align*}
		\frac{d^2}{dt^2}g_{i,j}&=\s{\D_{\w}\D_{\e_i}\w}{\e_j}+\s{\D_{\e_i}\w}{\D_{\w}\e_j}+\s{\D_{\w}\D_{\e_j}\w}{\e_i}+\s{\D_{\e_j}\w}{\D_{\w}\e_i}\\
		&=\s{R(\w,\e_i)\w}{\e_j}+\s{\D_{\e_i}\w}{\D_{\e_j}\w}+\s{R(\w,\e_j)\w}{\e_i}+\s{\D_{\e_i}\w}{\D_{\e_j}\w}\\
		&=2\s{\D_{\e_i}\w}{\D_{\e_j}\w}-2\s{R(\e_i,\w)\w}{\e_j}.
	\end{align*}
	Then, we compute
	\begin{align*}
	\frac{d^2}{dt^2}\det g_t&=g_{1,1}''g_{2,2}+2g_{1,1}'g_{2,2}'+g_{1,1}g_{2,2}''-2(g_{1,2}')^2\\
	&=2e^{2\lambda}\left(|\D_{\e_1}\w|^2+|\D_{\e_2}\w|^2-\s{R(\e_1,\w)\w}{\e_1}-\s{R(\e_2,\w)\w}{\e_2}\right)\\
	&+2\left(4\s{\D_{\e_1}\w}{\e_1}\s{\D_{\e_2}\w}{\e_2}-(\s{\D_{\e_1}\w}{\e_2}+\s{\D_{\e_2}\w}{\e_1})^2\right).
	\end{align*}
	Now, 
	\begin{align*}
		&(-1)^{i+j}\frac{d^2}{dt^2}g^{i,j}=e^{-4\lambda}\left(\frac{d^2}{dt^2}g_{i+1,j+1}\right)+2\left(\frac{d}{dt}g_{i+1,j+1}\right)\left(\frac{-\frac{d}{dt}\det g_t}{(\det g_t)^2}\right)+e^{2\lambda}\delta_{i,j}\left(\frac{-\frac{d^2}{dt^2}\det g_t}{(\det g_t)^2}+\frac{2\left(\frac{d}{dt}\det g_t\right)^2}{(\det g_t)^3}\right)\\
		&=e^{-4\lambda}\left(2\s{\D_{\e_{i+1}}\w}{\D_{\e_{j+1}}\w}-2\s{R(\e_{i+1},\w)\w}{\e_{j+1}}\right)\\
		&+2\left(\s{\D_{\e_{i+1}}\w}{\e_{j+1}}+\s{\D_{\e_{j+1}}\w}{\e_{i+1}}\right)(-e^{-8\lambda}(2e^{2\lambda}(\s{\D_{\e_1}\w}{\e_1}+\s{\D_{\e_2}\w}{\e_2})))\\
		&+e^{2\lambda}\delta_{i,j}\bigg\{-e^{-8\lambda}\Big(2e^{2\lambda}\left(|\D_{\e_1}\w|^2+|\D_{\e_2}\w|^2-\s{R(\e_1,\w)\w}{\e_1}-\s{R(\e_2,\w)\w}{\e_2}\right)\\
		&+2\left(4\s{\D_{\e_1}\w}{\e_1}\s{\D_{\e_2}\w}{\e_2}-(\s{\D_{\e_1}\w}{\e_2}+\s{\D_{\e_2}\w}{\e_1})^2\right)\Big)
		+2e^{-12\lambda}(4e^{4\lambda}\left(\s{\D_{\e_1}\w}{\e_1}+\s{\D_{\e_2}\w}{\e_2}\right)^2)\bigg\}.
	\end{align*}
	so for $(i,j)=(1,2)$, we have
	\begin{align*}
		-\frac{d^2}{dt^2}g_{1,2}&=2e^{-4\lambda}\left(\s{\D_{\e_1}\w}{\D_{\e_2}\w}-\s{R(\e_1,\w)\w}{\e_2}\right)-4e^{-6\lambda}\left(\s{\D_{\e_1}\w}{\e_2}+\s{\D_{\e_2}\w}{\e_1}\right)\left(\s{\D_{\e_1}\w}{\e_1}+\s{\D_{\e_2}\w}{\e_2}\right)
	\end{align*}
	and for $i=j$, writing $\D_i=e^{-\lambda}\D_{\e_i}$,
	\begin{align*}
		\frac{d^2}{dt^2}g^{i,i}&=2e^{-4\lambda}\left(|\D_{\e_{i+1}}\w|^2-\s{R(\e_{i+1},\w)\w}{\e_{i+1}}\right)-8e^{-6\lambda}\s{\D_{\e_{i+1}}\w}{\e_{i+1}}\left(\s{\D_{\e_1}\w}{\e_1}+\s{\D_{\e_2}\w}{\e_2}\right)\\
		&-2e^{-4\lambda}\left(|\D_{\e_1}\w|^2+|\D_{\e_2}\w|^2-\s{R(\e_1,\w)\w}{\e_1}-\s{R(\e_2,\w)\w}{\e_2}\right)\\
		&-e^{2\lambda}\left(4\s{\D_{\e_1}\w}{\e_1}\s{\D_{\e_2}\w}{\e_2}-\left(\s{\D_{\e_1}\w}{\e_1}+\s{\D_{\e_2}\w}{\e_2}\right)^2\right)+8e^{-6\lambda}\left(\s{\D_{\e_1}\w}{\e_1}+\s{\D_{\e_2}\w}{\e_2}\right)^2\\
		&=2e^{-4\lambda}\left(-|\D_{\e_i}\w|^2+\s{R(\e_i,\w)\w}{\e_i}\right)+8e^{-6\lambda}\left(\s{\D_{\e_i}\w}{\e_i}^2+\s{\D_{\e_1}\w}{\e_1}\s{\D_{\e_2}\w}{\e_2}\right)\\
		&-2e^{-6\lambda}\left(4\s{\D_{\e_1}\w}{\e_1}\s{\D_{\e_2}\w}{\e_2}-\left(\s{\D_{\e_1}\w}{\e_1}+\s{\D_{\e_2}\w}{\e_2}\right)^2\right)
	\end{align*}
	which gives the result by conformity, so
	\begin{align*}
		\frac{d^2}{dt^2}g^{i,j}&=2e^{-2\lambda}\left(-\s{\D_{\e_i}\w}{\D_{\e_j}\w}+\s{R(\e_i,\w)\w}{\e_j}\right)+4e^{-2\lambda}\left(\s{\D_{\e_i}\w}{\e_j}+\s{\D_{\e_j}\w}{\e_i}\right)\left(\s{\D_{\e_1}\w}{\e_1}+\s{\D_{\e_2}\w}{\e_2}\right)\\
		&-2\delta_{i,j} e^{-2\lambda}\left(4\s{\D_{\e_1}\w}{\e_1}\s{\D_{\e_2}\w}{\e_2}-\left(\s{\D_{\e_1}\w}{\e_1}+\s{\D_{\e_2}\w}{\e_2}\right)^2\right)
	\end{align*}
	which ends the proof of the lemma.
\end{proof}

\subsubsection{Second variation of the second fundamental form}

\begin{lemme}
	We have for $1\leq i,j\leq 2$
	\begin{align*}
	&\D_{\frac{d}{dt}}^\perp\D_{\frac{d}{dt}}^\perp \vec{\I}_{i,j}=\bigg(R(\w,\e_i)\D_{\e_j}\w+\D_{\e_i}R(\w,\e_j)\w+R(\D_{\e_i}\w,\e_j)\w+R(\w,\vec{\I}_{i,j})\w+R(\w,\e_j)\D_{\e_i}\w\\
	&+\D_{\w}R(\w,\e_i)\e_j+R(\w,\D_{\e_i}\w)\e_j+R(\w,\e_i)\D_{\e_j}\w
	-\sum_{k=1}^2\s{(\D_{\e_i}\D_{\e_j}-\D_{(\D_{\e_i}\e_j)^\top})\w+R(\w,\e_i)\e_j}{\e_k}\D_{\e_k}\w\bigg)^\perp.
	\end{align*}
\end{lemme}
\begin{proof}
We first recall that
\begin{align*}
	\D_{\frac{d}{dt}}^\perp\vec{\I}_{i,j}=\left((\D_{\e_i}\D_{\e_j}-\D_{\left(\D_{\e_i}\e_j\right)^\top})\w+R(\w,\e_i)\e_j\right)^\perp.
\end{align*}
Remembering \eqref{normalderivative}, we compute first
\begin{align*}
(\D_{\frac{d}{dt}}^{\perp}\D_{\e_i}\D_{\e_j}\w)^\perp&=R(\w,\e_i)\D_{\e_j}\w+\D_{\e_i}\left(\D_{\frac{d}{dt}}\D_{\e_j}\w\right)\\
&=R(\w,\e_i)\D_{\e_j}\w+\D_{\e_i}\left(R(\w,\e_j)\w+\D_{\e_j}\D_{\frac{d}{dt}}\w\right)\\
&=R(\w,\e_i)\D_{\e_j}\w+\D_{\e_i}\left(R(\w,\e_j)\w\right)\\
&=R(\w,\e_i)\D_{\e_j}\w+\D_{\e_i}R(\w,\e_j)\w+R(\D_{\e_i}\w,\e_j)\w+R(\w,\D_{\e_i}\e_j)\w+R(\w,\e_j)\D_{\e_i}\w
\end{align*}
Then,
\begin{align*}
\D_{\frac{d}{dt}}^\perp\D_{(\D_{\e_i}\e_j)^\top}\w&=R(\w,(\D_{\e_i}\e_j)^\top)\w+\D_{(\D_{\e_i}\e_j)^\top}\D_{\w}\w+\D_{\lie{(\D_{\e_i}\e_j)^\top}{\w}}\w\\
&=R(\w,(\D_{\e_i}\e_j)^\top)\w
\end{align*}
as at $t=0$,
\begin{align*}
	\lie{(\D_{\e_i}\e_j)^\top}{\w}&=\lie{\s{\D_{\e_i}\e_j}{\e_1}\e_1+\s{\D_{\e_i}\e_j}{\e_2}\e_2}{\w}\\
	&=\s{\D_{\e_i}\e_j}{\e_1}\lie{\e_1}{\w}+\s{\D_{\e_i}\e_j}{\e_2}\s{\e_2}{\w}
	+\left(d(\s{\D_{\e_i}\e_j}{\e_1})\cdot\e_1+d(\s{\D_{\e_i}\e_j}{\e_2})\cdot \e_2\right)\w\\
	&=\left(d(\s{\D_{\e_i}\e_j}{\e_1})\cdot\e_1+d(\s{\D_{\e_i}\e_j}{\e_2})\cdot \e_2\right)\w
\end{align*}
so
\begin{align*}
	\D_{\lie{(\D_{\e_i}\e_j)^\top}{\w}}\w=\left(d(\s{\D_{\e_i}\e_j}{\e_1})\cdot\e_1+d(\s{\D_{\e_i}\e_j}{\e_2})\cdot \e_2\right)\D_{\w}\w=0.
\end{align*}
and
\begin{align*}
\left(\D_{\frac{d}{dt}}\left((\D_{\e_i}\D_{\e_j}-\D_{(\D_{\e_i}\e_j)^\top})\w\right)\right)^\perp&=\bigg(R(\w,\e_i)\D_{\e_j}\w+\D_{\e_i}R(\w,\e_j)\w+R(\D_{\e_i}\w,\e_j)\w\\
&\quad\quad\quad\quad\quad\quad\quad\quad\quad+R(\w,\vec{\I}_{i,j})\w+R(\w,\e_j)\D_{\e_i}\w\bigg)^\perp
\end{align*}
while 
\begin{align*}
\D_{\frac{d}{dt}}\left(R(\w,\e_i)\e_j\right)=\D_{\w}R(\w,\e_i)\e_j+R(\w,\D_{\e_i}\w)\e_j+R(\w,\e_i)\D_{\e_j}\w
\end{align*}

so we get finally
\begin{align*}
&\D_{\frac{d}{dt}}^\perp\D_{\frac{d}{dt}}^\perp \vec{\I}_{i,j}=\bigg(R(\w,\e_i)\D_{\e_j}\w+\D_{\e_i}R(\w,\e_j)\w+R(\D_{\e_i}\w,\e_j)\w+R(\w,\vec{\I}_{i,j})\w+R(\w,\e_j)\D_{\e_i}\w\\
&+\D_{\w}R(\w,\e_i)\e_j+R(\w,\D_{\e_i}\w)\e_j+R(\w,\e_i)\D_{\e_j}\w
-\sum_{k=1}^2\s{(\D_{\e_i}\D_{\e_j}-\D_{(\D_{\e_i}\e_j)^\top})\w+R(\w,\e_i)\e_j}{\e_k}\D_{\e_k}\w\bigg)^\perp
\end{align*}
which concludes the proof.
\end{proof}

\subsubsection{Second variation of the mean curvature}

Now, making as earlier the shift of notation $\e_i=e^{-\lambda}\d{i}\phi$,
\begin{align*}
&2\D_{\frac{d}{dt}}^\perp\D_{\frac{d}{dt}}^\perp \vec{H}_{g_t}=\sum_{i,j=1}^2\left(\frac{d^2}{dt^2}g^{i,j}\right)\vec{\I}_{i,j}+2\left(\frac{d}{dt}g^{i,j}\right)\left(\D_{\frac{d}{dt}}^\perp \vec{\I}_{i,j}\right)+\delta_{i,j}e^{-2\lambda}\D_{\frac{d}{dt}}^\perp\D_{\frac{d}{dt}}^\perp\vec{\I}_{i,j}\\
&=\sum_{i,j=1}^2\bigg\{\left(2\left(-\s{\D_{\e_i}\w}{\D_{\e_j}\w}+\s{R(\e_i,\w)\w}{\e_j}\right)\right)\vec{\I}(\e_i,\e_j)\\
&+4\left(\left(\s{\D_{\e_i}\w}{\e_j}+\s{\D_{\e_j}\w}{\e_i}\right)\left(\s{\D_{\e_1}\w}{\e_1}+\s{\D_{\e_2}\w}{\e_2}\right)\right)\vec{\I}(\e_i,\e_j)\\
&-2\left(\s{\D_{\e_i}\w}{\e_j}+\s{\D_{\e_j}\w}{\e_i}\right)\left((\D_{\e_i}\D_{\e_j}-\D_{(\D_{\e_i}\e_j)^\top})\w+R(\w,\e_i)\e_j\right)^\perp\bigg\}\\
&-\sum_{i=1}^22\left(4\s{\D_{\e_1}\w}{\e_1}\s{\D_{\e_2}\w}{\e_2}-\left(\s{\D_{\e_1}\w}{\e_2}+\s{\D_{\e_1}\w}{\e_2}\right)^2\right)\vec{\I}(\e_i,\e_j)\\
&+\sum_{i=1}^2\bigg(R(\w,\e_i)\D_{\e_i}\w+\D_{\e_i}R(\w,\e_i)\w+R(\D_{\e_i}\w,\e_i)\w+R(\w,\vec{\I}(\e_i,\e_i))\w+R(\w,\e_i)\D_{\e_i}\w\\
&+\D_{\w}R(\w,\e_i)\e_i+R(\w,\D_{\e_i}\w)\e_i+R(\w,\e_i)\D_{\e_i}\w
-\sum_{k=1}^2\s{(\D_{\e_i}\D_{\e_i}-\D_{(\D_{\e_i}\e_i)^\top})\w+R(\w,\e_i)\e_i}{\e_k}\D_{\e_k}\w\bigg)^\perp
\end{align*}
Then
\begin{align*}
\frac{d^2}{dt^2}|\vec{H}_g|^2=2\s{\D_{\frac{d}{dt}}^\perp\D_{\frac{d}{dt}}^\perp\vec{H}_g}{H_g}+2\left|\D_{\frac{d}{dt}}^\perp\vec{H}_g\right|^2
\end{align*}
and
\begin{align*}
\frac{d^2}{dt^2}\int_{\Sigma}|H_{g_t}|^2d\mathrm{vol}_{g_t}&=\int_{\Sigma}2\s{\D_{\frac{d}{dt}}^\perp\D_{\frac{d}{dt}}^\perp\vec{H}_g}{H_g}+2\left|\D_{\frac{d}{dt}}^\perp\vec{H}_g\right|^2d\vg\\
&+4\int_{\Sigma}\s{\D_{\frac{d}{dt}}^\perp H_g}{\vec{H}_g}\left(\frac{d}{dt}d\vg\right)
+\int_{\Sigma}|\vec{H}_g|^2\left(\frac{d^2}{dt^2}d\vg\right).
\end{align*}
As
\begin{align*}
	\D_{\frac{d}{dt}}^\perp H_{g_t}=\frac{1}{2}\left(\Delta^{\n}_g\w+\mathscr{R}^\perp(\w)-\sum_{i,j=1}^2\left(\s{\D_{\e_i}\w}{\e_j}+\s{\D_{\e_j}\w}{\e_i}\right)\vec{\I}_{i,j}\right)
\end{align*}
we get
\begin{align}
	&D^2W(\phi)[\w,\w]=\int_{\Sigma}\bigg\langle\sum_{i,j=1}^2\bigg\{\left(2\left(-\s{\D_{\e_i}\w}{\D_{\e_j}\w}+\s{R(\e_i,\w)\w}{\e_j}\right)\right)\vec{\I}(\e_i,\e_j)\nonumber\\
	&+4\left(\s{\D_{\e_i}\w}{\e_j}+\s{\D_{\e_j}\w}{\e_i}\right)\left(\s{\D_{\e_1}\w}{\e_1}+\s{\D_{\e_2}\w}{\e_2}\right)\vec{\I}(\e_i,\e_j)\nonumber\\
		&-2\left(\s{\D_{\e_i}\w}{\e_j}+\s{\D_{\e_j}\w}{\e_i}\right)\left((\D_{\e_i}\D_{\e_j}-\D_{(\D_{\e_i}\e_j)^\top})\w+R(\w,\e_i)\e_j\right)^\perp\bigg\}\nonumber\\
		&-\sum_{i=1}^22\left(4\s{\D_{\e_1}\w}{\e_1}\s{\D_{\e_2}\w}{\e_2}-\left(\s{\D_{\e_1}\w}{\e_2}+\s{\D_{\e_1}\w}{\e_2}\right)^2\right)\vec{\I}(\e_i,\e_i)\nonumber\\
		&+\sum_{i=1}^2\bigg(R(\w,\e_i)\D_{\e_i}\w+\D_{\e_i}R(\w,\e_i)\w+R(\D_{\e_i}\w,\e_i)\w+R(\w,\vec{\I}(\e_i,\e_i))\w+R(\w,\e_i)\D_{\e_i}\w\nonumber\\
		&+\D_{\w}R(\w,\e_i)\e_i+R(\w,\D_{\e_i}\w)\e_i+R(\w,\e_i)\D_{\e_i}\w\nonumber\\
		&
		-\sum_{k=1}^2\s{(\D_{\e_i}\D_{\e_i}-\D_{(\D_{\e_i}\e_i)^\top})\w+R(\w,\e_i)\e_i}{\e_k}\D_{\e_k}\w\bigg)^\perp,\vec{H}_g\bigg\rangle d\vg\nonumber\\
        &+\frac{1}{2}\int_{\Sigma}\Big|\Delta^{\n}_g\w+\mathscr{R}_1^\perp(\w)-\sum_{i,j=1}^2\left(\s{\D_{\e_i}\w}{\e_j}+\s{\D_{\e_j}\w}{\e_i}\right)\vec{\I}(\e_i,\e_j)\Big|^2d\vg\nonumber\\
        &+2\int_{\Sigma}\s{\Delta^{\n}_g\w+\mathscr{R}_1^\perp(\w)-\sum_{i,j=1}^2\left(\s{\D_{\e_i}\w}{\e_j}+\s{\D_{\e_j}\w}{\e_i}\right)\vec{\I}(\e_i,\e_j)}{\vec{H}_g}\s{\D\w}{d\phi}d\vg\nonumber\\
        &+\int_{\Sigma}|\vec{H}_g|^2\left(|\D\w|^2-\mathscr{R}_2(\w,\w)-2\s{\D_{\e_1}\w}{\e_1}^2-2\s{\D_{\e_2}\w}{\e_2}^2-\left(\s{\D_{\e_1}\w}{\e_2}+\s{\D_{\e_2}\w}{\e_1}\right)^2\right)d\vg.
\end{align}
where
\begin{align*}
	\mathscr{R}_1^{\perp}=\left(\sum_{i=1}^2 R(\w,\e_i)\e_i\right)^\perp,\quad \mathscr{R}_2(\w,\w)=\sum_{i=1}^2\s{R(\e_i,\w)\w}{\e_i}
\end{align*}
We remark that this formula makes sense for $\phi\in W^{2,2}(\Sigma,N^3)\cap W^{1,\infty}(\Sigma,N^3)$, and $\w\in W^{2,2}(\Sigma,TN^3)\cap W^{1,\infty}(\Sigma,TN^3)$. This formula does not use the fact the $N$ is $3$-dimensional, and as mentioned above, it remains valid in every $C^3$ Riemannian manifold $(N^n,h)$ (this regularity is necessary, as $\D R$ is only continuous in a $C^3$ manifold).

In particular, for a minimal surface the equation takes the form
\begin{align*}
D^2W(\phi)[\n,\n]=\frac{1}{2}\int_{\Sigma}\Big|\Delta^{\n}_g\w+\mathscr{R}_1^\perp(\w)-\sum_{i,j=1}^2\left(\s{\D_{\e_i}\w}{\e_j}+\s{\D_{\e_j}\w}{\e_i}\right)\vec{\I}(\e_i,\e_j)\Big|^2d\vg
\end{align*}
and this shows the obvious fact that a minimal surface, which is an absolute minimiser of the Willmore functional, is stable.
For a normal variation, $\textit{i.e.}$ such that $\w=\pi_{\n}(\w)$, we have
\begin{align*}
	\s{\D_{\e_i}\w}{\e_j}=-\s{\w}{\D_{\e_i}\e_j}=-\s{\w}{\vec{\I}(\e_i,\e_j)}
\end{align*}
so
\begin{align*}
	D^2W(\phi)[\w,\w]=\frac{1}{2}\int_{\Sigma}\Big|\Delta^{\n}_g\w+\mathscr{R}_1^\perp(\w)+2\mathscr{A}(\w)\Big|^2d\vg
\end{align*}
where $\mathscr{A}$ is Simon's operator, defined by
\begin{align*}
	\A(\w)=\sum_{i,j=1}^{2}\s{\w}{\vec{\I}(\e_i,\e_j)}\vec{\I}(\e_i,\e_j).
\end{align*}
In the case of a surface with trivial normal bundle, this equation gets even simpler, as there exists $w\in W^{2,2}(\Sigma,\R)$, such that $\w=w\n$, and
\begin{equation}\label{d2minfacile}
	D^2W(\phi)[\w,\w]=\frac{1}{2}\int_{\Sigma}^{}\left(\Delta_gw+(|\vec{\I}_g|^2+\mathrm{Ric}_{N^3}(\n,\n))w\right)^2d\vg.
\end{equation}
Indeed, if $\D=\phi^\ast \D$ is the tangent connection defined on $\Sigma$, then we define the second order operator $\bar{\D}_{i,j}^2$ (see \cite{federer}, 5.4.12 for example), acting on smooth function on $\Sigma$, such that
\begin{align*}
	\bar{\D}_{i,j}^2=\bar{\D}_{\e_i}\bar{\D}_{\e_j}-\bar{\D}_{\bar{\D}_{\e_i}\e_j}
\end{align*}
If $f\in C^2(\Sigma)$, then we get the following expression for the laplacian $\Delta_g$ on $\Sigma$
\begin{align*}
	\Delta_gf=\mathrm{Tr}\bar{\D}^2f=\sum_{i=1}^2\bar{\D}_{i,i}^2f.
\end{align*}
We deduce that for a normal variation $\w=w\n$, as $(\D_{\e_i}\e_j)^\top=\bar{\D}_{\e_i}\e_j$, we have as $\n$ is a unit vector-field,
\begin{align*}
	\left(\left(\D_{\e_i}\D_{\e_j}-\D_{(\D_{\e_i}\e_j)^\top}\right)\w\right)^\perp&=\left(\bar{\D}^2_{i,j}w\right)\n+w\s{\D_{\e_i}\D_{\e_j}\n}{\n}\n\\
	&=\left(\bar{\D}^2_{i,j}w-w\s{\D_{\e_i}\n}{\D_{\e_j}\n}\right)\n\\
	&=\left(\bar{\D}^2_{i,j}w-we^{-4\lambda}\left(\I_{i,1}\I_{j,1}+\I_{i,2}\I_{j,2}\right)\right)\n
\end{align*}
so
\begin{align*}
	\Delta_g^{\n}\w=(\Delta_gw+w^2|\I_g|^2)\n
\end{align*}
while
\begin{align*}
	\mathscr{R}_2(\w,\w)=w^2\sum_{i=1}^2\s{R(\e_i,\n)\n}{\e_i}=w^2\,\mathrm{Ric}_{N^3}(\n,\n)
\end{align*}
So we also have
\begin{align}\label{d2vol}
	\frac{d^2}{dt^2}d\vg&=\left(|dw|_g^2+w^2|\vec{\I}_g|^2-w^2\,\mathrm{Ric}_{N^3}(\n,\n)-2w^2(\I(\e_1,\e_1)^2+\I(\e_2,\e_2)^2+2\I(\e_1,\e_2)^2)\right)d\vg\nonumber\\
	&=\left((|dw|_g^2-\left(|\vec{\I}_g|^2+\mathrm{Ric}_{N^3}(\n,\n)\right)w^2\right)d\vg
\end{align}
If $A$ the area functional. Then by \eqref{d2vol}, if $\phi$ is a minimal surface,
\begin{align*}
	D^2A(\phi)[\w,\w&]=\int_{\Sigma}\left(|dw|_g^2-\left(|\vec{\I}_g|^2+\mathrm{Ric}_{N^3}(\n,\n)\right)w^2\right)d\vg\\
	&=-\int_{\Sigma}w\left(\Delta_gw+\left(|\vec{\I}_g|^2+\mathrm{Ric}_{N^3}(\n,\n)\right)w\right)d\vg=-\int_{\Sigma}w\,L_gw\,d\vg
\end{align*}
so we have for a minimal surface $\phi$, if $L_g$ is the Jacobi operator of $\phi$,
\begin{align*}
	D^2W(\phi)[\w,\w]=\frac{1}{2}\int_{\Sigma}^{}(L_g w)^2d\vg.
\end{align*}
An other interesting case is the second variation of the conformal Willmore $\mathscr{W}=\W_{S^3}$ for a minimal surface in $S^3$, already present in the paper of Joel Weiner (\cite{weiner}) presenting first the Euler-Lagrange equation of Willmore functional in arbitrary Riemannian manifolds. If $\phi:\Sigma\rightarrow S^3$ is a minimal surface, then
\begin{align}\label{indicewillmin}
	D^2\W(\phi)[\w,\w]&=D^2A(\phi)[\w,\w]+D^2W(\phi)[\w,\w]\nonumber\\
	&=\int_{\Sigma}|dw|_g^2-\left(|\vec{\I}_g|^2+\mathrm{Ric}_{S^3}(\n,\n)\right)w^2d\vg+\frac{1}{2}\left(\Delta_gw+\left(|\vec{\I}_g|^2+\mathrm{Ric}_{S^3}(\n,\n)\right)w\right)^2d\vg\nonumber\\
	&=-\int_{\Sigma}^{}w\left(\Delta_gw+\left(|\vec{\I}_g|^2+2\right)w\right)d\vg+\frac{1}{2}\left(\Delta_gw+\left(|\vec{\I}_g|^2+2\right)w\right)^2d\vg\nonumber\\
	&=\frac{1}{2}\int_{\Sigma}^{}\left(\Delta_gw+\left(|\vec{\I}_g|^2+2\right)w\right)\left(\Delta_gw+|\vec{\I}_g|^2w\right)d\vg\nonumber\\
	&=\frac{1}{2}\int_{\Sigma}^{}w\left(L_g\circ (L_g-2)w\right)d\vg
\end{align}
so the index of a minimal surface in $S^3$ is equal to the (finite) number of negative eigenvalues of the strongly elliptic operator $L_g\circ (L_g-2)$. Therefore if $\lambda$ is a positive eigenvalue of $L_g$, $E_\lambda$ the eigenspace associated to $\lambda$ we define by
\begin{align*}
	\dim E_\lambda
\end{align*}
the dimension of the eigenspace. Therefore, we deduce that
\begin{equation}\label{formuleindicemin}
	\mathrm{Ind}_{\W}(\phi)=\sum_{0<\lambda<2}\dim E_\lambda.
\end{equation}
which was already contained in the paper of Joel Weiner $\cite{weiner}$ (note the different sign convention which we use here).

\section{First and second variation of Gauss curvature}\label{sectionK}

In this section, we compute the first and second variation of the Gauss curvature. This may seem at first useless according to the Gauss-Bonnet theorem, as for every closed surface $\Sigma$, for all smooth metric $g$ on $\Sigma$ we have
\begin{align*}
	\int_{\Sigma}K_gd\vg=2\pi\, \chi(\Sigma).
\end{align*}
However, this formula needs corrections for a non-closed surface. And when we perform variations, the total curvature will not be constant in general. The need to consider non-closed surfaces will be clarified in the next section, and the reader may first skip this technical part to get first some motivation for performing these computations.
We fix an arbitrary Riemann surface $\Sigma$, which is not supposed to be closed.

\subsection{First variation of K}

	\begin{lemme}
		For all smooth immersion $\phi:\Sigma\rightarrow\R^n$, for any admissible variation $\{\phi_t\}_{t\in I}$ of $\phi$, we have for all $t\in I$,
		\begin{align*}
		\frac{d}{dt}\left(K_{g_t}d\mathrm{vol}_{g_t}\right)=d\left(\vec{\I}\res_g \star d\w\right).
		\end{align*}
		where in local coordinates $(x_1,x_2)$, we have
		\begin{align*}
			\vec{\I}\res_g\star d\w=e^{-2\lambda}\sum_{i,j=1}^{2}(-1)^{i+j-1}\s{\vec{\I}_{i,j}}{\D_{\e_{j+1}}\w}\star dx_i
		\end{align*}
	\end{lemme}
	
	\begin{proof}
		We have in conformal coordinates
		\begin{align*}
		K_{g_t}d\mathrm{vol}_{g_t}=e^{-4\lambda(t)}\left(\s{\vec{\I}_{1,1}}{\vec{\I}_{2,2}}-|\vec{\I}_{1,2}|^2\right)e^{2\lambda(t)}dx_1\wedge dx_2,
		\end{align*}
		where we recall the notation $e^{4\lambda(t)}=\det g_t=|\d{1}\phi_t\wedge \d{2}\phi_t|^2$.
		Therefore we need to compute
		\begin{align}\label{derivative}
		\frac{d}{dt}\left((\det g_t)^{-\frac{1}{2}}\left(\s{\vec{\I}_{1,1}}{\vec{\I}_{2,2}}-|\vec{\I}_{1,2}|^2\right)\right)&=\left(\frac{d}{dt}(\det g_t)^{-\frac{1}{2}}\right)\left(\s{\vec{\I}_{1,1}}{\vec{\I}_{2,2}}-|\vec{\I}_{1,2}|^2\right)\nonumber\\
		&+\left(\det g_t\right)^{-\frac{1}{2}}\frac{d}{dt}\left(\s{\vec{\I}_{1,1}}{\vec{\I}_{2,2}}-|\vec{\I}_{1,2}|^2\right)\\
		&=\mathrm{(I)}+\mathrm{(II)}\nonumber
		\end{align}
		We have by lemma \ref{lemme1},
		\begin{align*}
		\frac{d}{dt}\left(\det g_t\right)^{-\frac{1}{2}}&=-\frac{1}{2}e^{-6\lambda}\left(2e^{2\lambda_2}\s{\D_{\e_1}\w}{\e_1}+2e^{2\lambda_1}\s{\D_{\e_2}\w}{\e_2}-2\mu\left(\s{\D_{\e_1}\w}{\e_2}+\s{\D_{\e_2}\w}{\e_1}\right)\right)\\
		&=-e^{-6\lambda}\left(e^{2\lambda_2}\s{\D_{\e_1}\w}{\e_1}+e^{2\lambda_1}\s{\D_{\e_2}\w}{\e_2}-\mu\left(\s{\D_{\e_1}\w}{\e_2}+\s{\D_{\e_2}\w}{\e_1}\right)\right)
		\end{align*}
		and
		\begin{equation}\label{detgprime}
		\mathrm{(II)}=\left(\frac{d}{dt}(\det g_t)^{-\frac{1}{2}}\right)K_{g_t}=-e^{-2\lambda}\left(e^{2\lambda_2}\s{\D_{\e_1}\w}{\e_1}+e^{2\lambda_1}\s{\D_{\e_2}\w}{\e_2}-\mu\left(\s{\D_{\e_1}\w}{\e_2}+\s{\D_{\e_2}\w}{\e_1}\right)\right)K_{g_t}
		\end{equation}
		Therefore
		\begin{align*}
		\D_{\frac{d}{dt}}^\perp \vec{\I}_{i,j}=\left(\left(\D_{\e_i}\D_{\e_j}-\D_{(\D_{\e_i}\e_j)^\top}\right)\w\right)^\perp
		\end{align*}
		Now, recall that $e^{2\lambda_i(t)}=|\d{i}\phi_t|^2$, and we introduce
		\begin{align*}
		\alpha_{i,j}^k(t)=\s{\D_{\e_i(t)}\e_j(t)}{\e_k(t)}=\s{\partial_{x_i,x_j}^2\phi_t}{\d{k}\phi_t}
		\end{align*}
		and to simplify the writing, we drop the $t$ index in the following. We recall that
		\begin{align}\label{projtang}
		(\D_{\e_i}\e_j)^\top&=e^{-4\lambda}\left(e^{2\lambda_2}\s{\D_{\e_i}\e_j}{\e_1}-\mu\s{\D_{\e_i}\e_j}{\e_2}\right)\e_1\nonumber\\
		&+e^{-4\lambda}\left(e^{2\lambda_1}\s{\D_{\e_i}\e_j}{\e_2}-\mu\s{\D_{\e_i}\e_j}{\e_1}\right)\e_2\nonumber\\
		&=e^{-4\lambda}\left(e^{2\lambda_2}\alpha_{i,j}^1-\mu\,\alpha_{i,j}^2\right)\e_1+e^{-4\lambda}\left(e^{2\lambda_1}\alpha_{i,j}^2-\mu\,\alpha_{i,j}^1\right)\e_2\nonumber\\
		&=e^{-4\lambda}\sum_{l=1}^2\left(e^{2\lambda_{l+1}}\alpha_{i,j}^{l}-\mu\,\alpha_{i,j}^{l+1}\right)\e_l
		\end{align} 
		\begin{align*}
		\frac{d}{dt}\s{\vec{\I}_{1,1}}{\vec{\I}_{2,2}}&=\s{\D_{\e_1}^2\w-e^{-4\lambda}\sum_{l=1}^2\left(e^{2\lambda_{l+1}}\alpha_{1,1}^{l}-\mu\,\alpha_{1,1}^{l+1}\right)\D_{\e_l}\w}{\vec{\I}_{2,2}}\\
		&+\s{\D_{\e_2}^2\w-e^{-4\lambda}\sum_{l=1}^2\left(e^{2\lambda_{l+1}}\alpha_{2,2}^{l}-\mu\,\alpha_{2,2}^{l+1}\right)\D_{\e_l}\w}{\vec{\I}_{1,1}}
		\end{align*}
		and
		\begin{align*}
		\frac{d}{dt}|\vec{\I}_{1,2}|^2=2\s{\D_{\e_1}\D_{\e_2}\w-e^{-4\lambda}\sum_{l=1}^2\left(e^{2\lambda_{l+1}}\alpha_{1,2}^{l}-\mu\,\alpha_{1,2}^{l+1}\right)\D_{\e_l}\w}{\vec{\I}_{1,2}}.
		\end{align*}
		so
		\begin{align}
		\mathrm{(II)}&=e^{-2\lambda}\left(\s{\D_{\e_1}^2\w}{\vec{\I}_{2,2}}+\s{\D_{\e_2}\w}{\vec{\I}_{1,1}}-2\s{\D_{\e_1}\D_{\e_2}\w}{\vec{\I}_{1,2}}\right)\label{II1}\\
		&=-e^{-6\lambda}\sum_{l=1}^{2}\left(\left(e^{2\lambda_{l+1}}\alpha_{1,1}^l-\mu\alpha_{1,1}^{l+1}\right)\s{\D_{\e_l}\w}{\vec{\I}_{2,2}}
		+\left(e^{2\lambda_{l+1}}\alpha_{2,2}^l-\mu\alpha_{2,2}^{l+1}\right)\s{\D_{\e_l}\w}{\vec{\I}_{1,1}}\right)\nonumber\\
		&+2e^{-6\lambda}\sum_{l=1}^{2}\left(e^{2\lambda_{l+1}}\alpha_{1,2}^l-\mu\alpha_{1,2}^{l+1}\right)\s{\D_{\e_l}\w}{\vec{\I}_{1,2}}\label{II2}
		\end{align}
		We first compute \eqref{II1}, and we make some remarks first on covariant derivative.	We recall that the covariant derivative $\bar{\D}$ is just the orthogonal projection on $T\R^3\rightarrow \phi_{t\ast}(T\Sigma)$.
		By the definition of the covariant derivative $\D$, if $X,Y,Z$ are tangent vectors, then
		\begin{align*}
		\D^\perp_X \vec{\I}(Y,Z)=\D_X^\perp(\vec{\I}(Y,Z))-\vec{\I}(\bar{\D}_XY,Z)-\vec{\I}(Y,\bar{\D}_XZ)
		\end{align*}
		while Codazzi-Mainardi identity reads as $R=0$
		\begin{align*}
		\D_{X}^\perp \vec{\I}(Y,Z)=\D_{Y}^\perp \vec{\I}(X,Z).
		\end{align*}
		Therefore, by compatibility with the metric, we get for $i=1,2$
		\begin{align*}
		e^{-\lambda}\s{\D_{\e_k}^2\w}{\vec{\I}_{i,j}}&=\D_{\e_k}\left(\s{\D_{\e_k}\w}{e^{-2\lambda}\vec{\I}_{i,j}}\right)-e^{-2\lambda}\s{\D_{\e_k}\w}{\D_{\e_i}\left(\vec{\I}_{i,j}\right)}-\d{k}e^{-2\lambda}\s{\D_{\e_k}\w}{\vec{\I}_{i,j}}
		\end{align*}
		and invoking the two preceding identities, we get using by orthogonality of $\vec{\I}_{i,j}$ for the second line,
		\begin{align*}
		\D_{\e_k}^\top\left(\vec{\I}_{i,j}\right)&=e^{-4\lambda}\left(e^{2\lambda_2}\s{\D_{\e_k}\left(\vec{\I}_{i,j}\right)}{\e_1}-\mu\s{\D_{\e_k}\left(\vec{\I}_{i,j}\right)}{\e_2}\right)\e_1\\
		&+e^{-4\lambda}\left(e^{2\lambda_1}\s{\D_{\e_k}\left(\vec{\I}_{i,j}\right)}{\e_2}-\mu\s{\D_{\e_k}\left(\vec{\I}_{i,j}\right)}{\e_1}\right)\e_2\\
		&=-e^{-4\lambda}\sum_{l=1}^2\left(e^{2\lambda_{l+1}}\s{\vec{\I}_{i,j}}{\vec{\I}_{k,l}}-\mu\s{\vec{\I}_{i,j}}{\vec{\I}_{k,l+1}}\right)\e_l
		\end{align*}
		The, by definition of $\D$, and Codazzi-Mainardi identity, we get
		\begin{align*}
		\D_{\e_k}\left(\vec{\I}_{i,j}\right)&=\D_{\e_k}^\perp\left(\vec{\I}_{i,j}\right)+\D_{\e_k}^\top\left(\vec{\I}_{i,j}\right)\\
		&=\D_{\e_k}^\perp\vec{\I}(\e_i,\e_j)+\vec{\I}(\bar{\D}_{\e_k}\e_i,\e_j)+\vec{\I}(\e_i,\bar{\D}_{\e_k}\e_j)-e^{-4\lambda}\sum_{l=1}^2\left(e^{2\lambda_{l+1}}\s{\vec{\I}_{i,j}}{\vec{\I}_{k,l}}-\mu\s{\vec{\I}_{i,j}}{\vec{\I}_{k,l+1}}\right)\e_l\\
		&=\D_{\e_i}^\perp\vec{\I}(\e_k,\e_j)+e^{-4\lambda}\sum_{l=1}^2\left(e^{2\lambda_{l+1}}\alpha_{i,k}^l-\mu\alpha_{i,k}^{l+1}\right)\vec{\I}_{j,l}
		+e^{-4\lambda}\sum_{l=1}^{2}\left(e^{2\lambda_{l+1}}\alpha_{j,k}^{l}-\mu\alpha_{j,k}^{l+1}\right)\vec{\I}_{i,l}\\
		&-e^{-4\lambda}\sum_{l=1}^2\left(e^{2\lambda_{l+1}}\s{\vec{\I}_{i,j}}{\vec{\I}_{k,l}}-\mu\s{\vec{\I}_{i,j}}{\vec{\I}_{k,l+1}}\right)\e_l
		\end{align*}
		and now
		\begin{align*}
		\D_{\e_i}^\perp\vec{\I}(\e_j,\e_k)&=\D_{\e_i}\left(\vec{\I}_{j,k}\right)-\D_{\e_i}^\top\left(\vec{\I}_{j,k}\right)-\vec{\I}(\bar{\D}_{\e_i}\e_j,\e_k)-\vec{\I}(\e_j,\bar{\D}_{\e_i}\e_k)\\
		&=\D_{\e_i}^\perp\left(\vec{\I}_{j,k}\right)-e^{-4\lambda}\sum_{l=1}^{2}\left(e^{2\lambda_{l+1}}\alpha_{i,j}^l-\mu\alpha_{i,j}^{l+1}\right)\vec{\I}_{k,l}-e^{-4\lambda}\sum_{l=1}^{2}\left(e^{2\lambda_{l+1}}\alpha_{i,k}^l-\mu\alpha_{i,k}^{l+1}\right)\vec{\I}_{j,l}\\
		&+e^{-4\lambda}\sum_{l=1}^{2}\left(e^{2\lambda_{l+1}}\s{\vec{\I}_{j,k}}{\vec{\I}_{i,l}}-\mu\s{\vec{\I}_{j,k}}{\vec{\I}_{i,l+1}}\right)\e_l
		\end{align*}
		so summing both expression, two sums cancel, and we get
		\begin{align}\label{codazzi}
		\D_{\e_k}\left(\vec{\I}_{i,j}\right)&=\D_{\e_i}\left(\vec{\I}_{i,j}\right)+e^{-4\lambda}\sum_{l=1}^2\left(e^{2\lambda_{l+1}}\alpha_{j,k}^l-\mu\alpha_{j,k}^{l+1}\right)\vec{\I}_{i,l}-e^{-4\lambda}\sum_{l=1}^{2}\left(e^{2\lambda_{l+1}}\alpha_{i,j}^l-\mu\alpha_{i,j}^{l+1}\right)\vec{\I}_{k,l}\nonumber\\
		&+e^{-4\lambda}\sum_{l=1}^2\left(e^{2\lambda_{l+1}}\left(\s{\vec{\I}_{j,k}}{\vec{\I}_{i,l}}-\s{\vec{\I}_{i,j}}{\vec{\I}_{k,l}}\right)-\mu\left(\s{\vec{\I}_{j,k}}{\vec{\I}_{i,l+1}}-\s{\vec{\I}_{i,j}}{\vec{\I}_{k,l+1}}\right)\right)\e_l
		\end{align}
		We deduce that
		\begin{align*}
		&e^{-2\lambda}\s{\D_{\e_1}^2\w}{\vec{\I}_{2,2}}+e^{-2\lambda}\s{\D_{\e_2}^2\w}{\vec{\I}_{1,1}}-2\s{\D_{\e_1}\D_{\e_2}\w}{\vec{\I}_{1,2}}\\
		&=\D_{\e_1}\left(\s{\D_{\e_1}\w}{e^{-2\lambda}\vec{\I}_{2,2}}\right)
		+\D_{\e_2}\left(\s{\D_{\e_2}\w}{e^{-2\lambda}\vec{\I}_{1,1}}\right)
		-\D_{\e_1}\left(\s{\D_{\e_2}\w}{e^{-2\lambda}\vec{\I}_{1,2}}\right)
		-\D_{\e_2}\left(\s{\D_{\e_1}\w}{e^{-2\lambda}\vec{\I}_{1,2}}\right)\\
		&-e^{-2\lambda}\s{\D_{\e_1}\w}{\D_{\e_1}\left(\vec{\I}_{2,2}\right)}
		-e^{-2\lambda}\s{\D_{\e_2}\w}{\D_{\e_2}\left(\vec{\I}_{1,1}\right)}
		+e^{-2\lambda}\s{\D_{\e_1}\w}{\D_{\e_2}\left(\vec{\I}_{1,2}\right)}
		+e^{-2\lambda}\s{\D_{\e_2}\w}{\D_{\e_1}\left(\vec{\I}_{1,2}\right)}\\
		&-\d{1}e^{-2\lambda}\s{\D_{\e_1}\w}{\vec{\I}_{2,2}}
		-\d{2}e^{-2\lambda}\s{\D_{\e_1}\w}{\vec{\I}_{1,1}}
		+\d{1}e^{-2\lambda}\s{\D_{\e_2}\w}{\vec{\I}_{1,2}}
		+\d{2}e^{-2\lambda}\s{\D_{\e_1}\w}{\vec{\I}_{1,2}}\\
		&=\mathrm{(i)}+\mathrm{(ii)}+\mathrm{(iii)}
		\end{align*}
		where $\mathrm{(i)}$, $\mathrm{(ii)}$ and $\mathrm{(iii)}$ correspond to the first, second and third line respectively on the preceding equality.
		Now, 
		\begin{align*}
		\mathrm{(i)}dx_1\wedge dx_2=d\left(\vec{\I}\res_g \star d\w\right).
		\end{align*}
		Indeed, we have
		\begin{align*}
			&\D_{\e_1}\left(\s{\D_{\e_1}\w}{e^{-2\lambda}\vec{\I}_{2,2}}\right)
			+\D_{\e_2}\left(\s{\D_{\e_2}\w}{e^{-2\lambda}\vec{\I}_{1,1}}\right)
			-\D_{\e_1}\left(\s{\D_{\e_2}\w}{e^{-2\lambda}\vec{\I}_{1,2}}\right)
			-\D_{\e_2}\left(\s{\D_{\e_1}\w}{e^{-2\lambda}\vec{\I}_{1,2}}\right)\\
			&=d\left(e^{-2\lambda}\left(-\s{\vec{\I}_{1,1}}{\D_{\e_2}\w}+\s{\vec{\I}_{1,2}}{\D_{\e_1}\w}\right)dx_2+e^{-2\lambda}\left(\s{\vec{\I}_{2,2}
				}{\D_{\e_1}\w}-\s{\vec{\I}_{2,1}}{\D_{\e_2}\w}\right)dx_1\right)\\
			&=d\left(\vec{\I}\res_g\star d\w=e^{-2\lambda}\sum_{i,j=1}^{2}(-1)^{i+j-1}\s{\vec{\I}_{i,j}}{\D_{\e_{j+1}}\w}\star dx_i\right)
		\end{align*}
		Therefore, we simply need to verify that all remaining terms cancel.
		Using the Codazzi-Mainardi identity \eqref{codazzi} for the first two terms of $\mathrm{(ii)}$ cancels the last two terms and we get
		\begin{align*}
		\mathrm{(ii)}&=-e^{-2\lambda}\s{\D_{\e_1}\w}{e^{-4\lambda}\sum_{l=1}^2\left(e^{2\lambda_{l+1}}\alpha_{12}^l-\mu\alpha_{12}^{l+1}\right)\vec{\I}_{2,l}-e^{-4\lambda}\sum_{l=1}^{2}\left(e^{2\lambda_{l+1}}\alpha_{2,2}^l-\mu\alpha_{2,2}^{l+1}\right)\vec{\I}_{1,l}}\\
		&-e^{-2\lambda}\s{\D_{\e_1}\w}{e^{-4\lambda}\sum_{l=1}^2\left(e^{2\lambda_{l+1}}\left(\s{\vec{\I}_{1,2}}{\vec{\I}_{2,l}}-\s{\vec{\I}_{2,2}}{\vec{\I}_{1,l}}\right)-\mu\left(\s{\vec{\I}_{1,2}}{\vec{\I}_{2,l+1}}-\s{\vec{\I}_{2,2}}{\vec{\I}_{1,l+1}}\right)\right)\e_l}\\
		&-e^{-2\lambda}\s{\D_{\e_2}\w}{e^{-4\lambda}\sum_{l=1}^2\left(e^{2\lambda_{l+1}}\alpha_{1,2}^l-\mu\alpha_{1,2}^{l+1}\right)\vec{\I}_{1,l}-e^{-4\lambda}\sum_{l=1}^{2}\left(e^{2\lambda_{l+1}}\alpha_{1,1}^l-\mu\alpha_{1,1}^{l+1}\right)\vec{\I}_{2,l}}\\
		&-e^{-2\lambda}\s{\D_{\e_2}\w}{e^{-4\lambda}\sum_{l=1}^2\left(e^{2\lambda_{l+1}}\left(\s{\vec{\I}_{1,2}}{\vec{\I}_{1,l}}-\s{\vec{\I}_{1,1}}{\vec{\I}_{2,l}}\right)-\mu\left(\s{\vec{\I}_{1,2}}{\vec{\I}_{1,l+1}}-\s{\vec{\I}_{1,1}}{\vec{\I}_{2,l+1}}\right)\right)\e_l}
		\end{align*}
		The sum of the odd lines of this expression gives
		\begin{align*}
		&-e^{-2\lambda}\s{\D_{\e_1}\w}{e^{-4\lambda}\left(e^{2\lambda_2}(|\vec{\I}_{1,2}|^2-\s{\vec{\I}_{2,2}}{\vec{\I}_{1,1}})\e_1-\mu((|\vec{\I}_{1,2}|^2-\s{\vec{\I}_{2,2}}{\vec{\I}_{1,1}}))\e_2\right)}\\
		&-e^{-2\lambda}\s{\D_{\e_2}\w}{e^{-4\lambda}\left(e^{2\lambda_1}(|\vec{\I}_{1,2}|^2-\s{\vec{\I}_{1,1}}{\vec{\I}_{2,2}})\e_2-\mu((|\vec{\I}_{1,2}|^2-\s{\vec{\I}_{1,1}}{\vec{\I}_{2,2}}))\e_1\right)}\\
		&=e^{-2\lambda}\left(e^{2\lambda_2}\s{\D_{\e_1}\w}{\e_1}+e^{2\lambda_1}\s{\D_{\e_2}\w}{\e_2}-\mu\left(\s{\D_{\e_1}\w}{\e_2}+\s{\D_{\e_2}\w}{\e_1}\right)\right)K_{g_t}\\
		&=-\left(\frac{d}{dt}(\det g_t)^{-\frac{1}{2}}\right)K_{g_t}
		\end{align*}
		which cancels with \eqref{detgprime}.
		Now, we have
		\begin{align*}
		\d{i}e^{4\lambda}=2\left(\alpha_{i,1}^1e^{2\lambda_2}+\alpha_{i,2}^2e^{2\lambda_1}-\mu\left(\alpha_{i,1}^2+\alpha_{i,2}^1\right)\right)
		\end{align*}
		so
		\begin{align*}
		\d{i}e^{-2\lambda}=-e^{-6\lambda}\left(\alpha_{i,1}^1e^{2\lambda_2}+\alpha_{i,2}^2e^{2\lambda_1}-\mu\left(\alpha_{i,1}^2+\alpha_{i,2}^1\right)\right).
		\end{align*}
		The remaining terms is the sum of the even lines of the last expression of (ii), (iii) and \eqref{II2}:
		\begin{align*}
		&-e^{-6\lambda}\sum_{l=1}^{2}\left(e^{2\lambda_{l+1}}\alpha_{1,2}^l-\mu\alpha_{1,2}^{l+1}\right)\s{\D_{\e_1}\w}{\vec{\I}_{2,l}}+e^{-6\lambda}\sum_{l=1}^{2}\left(e^{2\lambda_{l+1}}\alpha_{2,2}^l-\mu\alpha_{2,2}^{l+1}\right)\s{\D_{\e_1}\w}{\vec{\I}_{1,l}}\\
		&-e^{-6\lambda}\sum_{l=1}^{2}\left(e^{2\lambda_{l+1}}\alpha_{1,2}^l-\mu\alpha_{1,2}^{l+1}\right)\s{\D_{\e_2}\w}{\vec{\I}_{1,l}}+e^{-6\lambda}\sum_{l=1}^{2}\left(e^{2\lambda_{l+1}}\alpha_{1,1}^l-\mu\alpha_{1,1}^{l+1}\right)\s{\D_{\e_2}\w}{\vec{\I}_{2,l}}\\
		&+e^{-6\lambda}\left(\alpha_{1,1}^1e^{2\lambda_2}+\alpha_{1,2}^2e^{2\lambda_1}-\mu\left(\alpha_{1,1}^2+\alpha_{1,2}^1\right)\right)\s{\D_{\e_1}\w}{\vec{\I}_{2,2}}\\
		&+e^{-6\lambda}\left(\alpha_{1,2}^1e^{2\lambda_2}+\alpha_{2,2}^2e^{2\lambda_1}-\mu\left(\alpha_{2,1}^2+\alpha_{2,2}^1\right)\right)\s{\D_{\e_2}\w}{\vec{\I}_{1,1}}\\
		&-e^{-6\lambda}\left(\alpha_{1,1}^1e^{2\lambda_2}+\alpha_{1,2}^2e^{2\lambda_1}-\mu\left(\alpha_{1,1}^2+\alpha_{1,2}^1\right)\right)\s{\D_{\e_2}\w}{\vec{\I}_{1,2}}\\
		&-e^{-6\lambda}\left(\alpha_{1,2}^1e^{2\lambda_2}+\alpha_{2,2}^2e^{2\lambda_1}-\mu\left(\alpha_{2,1}^2+\alpha_{2,2}^1\right)\right)\s{\D_{\e_1}\w}{\vec{\I}_{1,2}}\\
		&-e^{-6\lambda}\sum_{l=1}^{2}\left(\left(e^{2\lambda_{l+1}}\alpha_{1,1}^l-\mu\alpha_{1,1}^{l+1}\right)\s{\D_{\e_l}\w}{\vec{\I}_{2,2}}
		+\left(e^{2\lambda_{l+1}}\alpha_{2,2}^l-\mu\alpha_{2,2}^{l+1}\right)\s{\D_{\e_l}\w}{\vec{\I}_{1,1}}\right)\nonumber\\
		&+2e^{-6\lambda}\sum_{l=1}^{2}\left(e^{2\lambda_{l+1}}\alpha_{1,2}^l-\mu\alpha_{1,2}^{l+1}\right)\s{\D_{\e_l}\w}{\vec{\I}_{1,2}}\\
		&=-\ala{2}{1,2}{1}{1}{1,2}-\ala{1}{1,2}{2}{1}{2,2}\\
		&+\ala{2}{2,2}{1}{1}{1,1}+\ala{1}{2,2}{2}{1}{1,2}\\
		&-\ala{2}{1,2}{1}{2}{1,1}-\ala{1}{1,2}{2}{2}{1,2}\\
		&+\ala{2}{1,1}{1}{2}{1,2}+\ala{1}{1,1}{2}{2}{2,2}\\
		&+\ala{2}{1,1}{1}{1}{2,2}+\ala{1}{1,2}{2}{1}{2,2}\\
		&+\ala{2}{1,2}{1}{2}{1,1}+\ala{1}{2,2}{2}{2}{1,1}\\
		&-\ala{2}{1,1}{1}{2}{1,2}-\ala{1}{1,2}{2}{2}{1,2}\\
		&-\ala{2}{1,2}{1}{1}{1,2}-\ala{1}{2,2}{2}{1}{1,2}\\
		&-\ala{2}{1,1}{1}{1}{2,2}-\ala{1}{1,1}{2}{2}{2,2}\\
		&-\ala{2}{2,2}{1}{1}{1,1}-\ala{1}{2,2}{2}{2}{1,1}\\
		&+2\ala{2}{1,2}{1}{1}{1,2}+2\ala{1}{1,2}{2}{2}{1,2}\\
		&=0
		\end{align*}
		which concludes the proof of the lemma.
	\end{proof}

\subsection{Second variation of K}

\begin{lemme}\label{d2k}
	For every minimal immersion $\phi:\Sigma\rightarrow\R^3$, and for all normal admissible variation $\{\phi_t\}_{t\in I}$ of $\phi$ of variation vector $\w=w\n$, we have
	\begin{equation}
		\frac{d^2}{dt^2}\left(K_{g_t}d\mathrm{vol}_{g_t}\right)_{|t=0}=d\left(\left(\Delta_g w+2Kw\right)\star dw-\frac{1}{2}\star d|dw|_g^2\right).
	\end{equation}
\end{lemme}
\begin{proof}Thanks of lemma, we have
	\begin{align*}
	\frac{d}{dt}(K\,d\vg)=d(\vec{\I}\res_g \star d\w)
	\end{align*}
	which reads in local coordinates
	\begin{align*}
	\vec{\I}\res\star d\w=e^{-2\lambda}\left(-\s{\vec{\I}_{1,1}}{\partial_{x_2}\w}+\s{\vec{\I}_{1,2}}{\partial_{x_1}\w}\right)dx_1+e^{-2\lambda}\left(\s{\vec{\I}_{2,2}}{\partial_{x_1}\w}-\s{\vec{\I}_{1,2}}{\partial_{x_2}\w}\right)dx_2.
	\end{align*}
	where $e^{4\lambda}=e^{4\lambda(t)}=\det g_t=|\d{1}\phi_t\wedge\d{2}\phi_t|$.
	We have as $\w=w\n$,
	\begin{align*}
	e^{-2\lambda}\s{\vec{\I}_{i,j}}{\d{k}\w}=e^{-2\lambda}\I_{i,j}\s{\n}{\d{k}\w}=e^{-2\lambda}\s{\partial^2_{x_i,x_j}\phi}{\n}\d{k}w
	\end{align*}
	Now,
	\begin{align*}
	{\frac{d}{dt}e^{2\lambda}}_{|t=0}=\s{\d{1}\phi}{\d{1}\w}+\s{\d{2}\phi}{\d{2}\w}=\s{\d{1}\phi}{\d{1}\n}w+\s{\d{2}\phi}{\d{2}\n}w=-2e^{2\lambda}wH_g=0
	\end{align*}
	as $M=\phi(\Sigma)$ is minimal. So we have
	\begin{align*}
	{\frac{d}{dt}\left(e^{-2\lambda}\s{\vec{\I}_{i,j}}{\d{k}\w}\right)}_{|t=0}=e^{-2\lambda}{\frac{d}{dt}\left(\s{\partial^2_{x_i,x_j}\phi}{\n}\s{\d{k}\w}{\n}\right)}_{|t=0}.
	\end{align*}
	And choosing conformal coordinates, we compute
	\begin{align*}
	&{\frac{d}{dt}\s{\partial^2_{x_i,x_j}\phi}{\n}}_{|t=0}=\s{\partial^2_{x_i,x_j}\w}{\n}+\s{\partial^2_{x_i,x_j}\phi}{-\sum_{l=1}^2e^{-2\lambda}\s{\d{l}\w}{\n}\d{l}\phi}\\
	&=\s{(\partial^2_{x_i,x_j}w)\n+\d{i}w\d{j}\n+\d{j}w\d{i}\n+w\partial^2_{x_i,x_j}\n}{\n}
	-\sum_{l=1}^2e^{-2\lambda}\s{\partial^2_{x_i,x_j}\phi}{\d{l}\phi}\d{l}w\\
	&=\partial^2_{x_i,x_j}w-w\s{\d{i}\n}{\d{j}\n}-e^{-2\lambda}\sum_{l=1}^2 \s{\partial^2_{x_i,x_j}\phi}{\d{l}\phi}\d{l}w
	\end{align*}
	\begin{align*}
	{\frac{d}{dt}\s{\d{k}\w}{\n}}_{|t=0}&=\s{\d{k}\w}{\frac{d}{dt}\n_{|t=0}}
	=\sum_{l=1}^{2}\s{\d{k}\w}{-e^{-2\lambda}\s{\n}{\d{l}\w}\d{l}\phi}\\
	&=-e^{-2\lambda}\sum_{l=1}^2\s{\d{k}\n}{\d{l}\phi}w\,\d{l}w=\sum_{l=1}^2\I_{k,l}\;w\,\d{l}w
	\end{align*}
	So
	\begin{align*}
	{\frac{d}{dt}\left(\s{\partial^2_{x_i,x_j}\phi}{\n}\s{\d{k}\w}{\n}\right)}_{|t=0}&=\left(\partial^2_{x_i,x_j}w-w\s{\d{i}\n}{\d{j}\n}-e^{-2\lambda}\sum_{l=1}^2 \s{\partial^2_{x_i,x_j}\phi}{\d{l}\phi}\d{l}w\right)\d{k}w\\
	&+e^{2\lambda}\I_{i,j}\sum_{l=1}^2\I_{k,l}\;w\,\d{l}w
	\end{align*}
	and finally,
	\begin{align*}
	e^{-2\lambda}{\frac{d}{dt}\s{\vec{\I}_{i,j}}{\d{k}\w}}_{|t=0}&=e^{-2\lambda}(\partial^2_{x_i,x_j}w)(\d{k}w)-e^{-2\lambda}w\s{\d{i}\n}{\d{j}\n}\d{k}w\\
	&-e^{-4\lambda}\sum_{l=1}^2\s{\partial^2_{x_i,x_j}\phi}{\d{l}\phi}(\d{k}w)(\d{l}w)
	+w\,\sum_{l=1}^2\I_{i,j}\I_{k,l}\d{l}w.
	\end{align*}
	By minimality, we have
	\begin{align*}
	|\D_{e_1}\n|_g^2=|\D_{\e_2}\n|_g^2=\frac{|\vec{\I}_g|^2}{2}=-K_g, \quad \s{\D_{\e_1}\n}{\D_{\e_2}\n}=0
	\end{align*}
	
	As a consequence, we get
	\begin{align*}
	&e^{-2\lambda}\frac{d}{dt}\left(-\s{\vec{\I}_{1,1}}{\d{2}\w}+\s{\vec{\I}_{1,2}}{\d{1}\w}\right)=-\left(e^{-2\lambda}\partial^2_{x_1}w-we^{-2\lambda}|\d{1}\n|^2\right)\d{2}w\\
	&+e^{-4\lambda}\left(\s{\partial^2_{x_1}\phi}{\d{1}\phi}\d{1}w+\s{\partial^2_{x_1}\phi}{\d{2}\phi}\d{2}w\right)\d{2}w
	-w(\I_{1,1}\I_{1,2}\d{1}w+\I_{1,1}\I_{2,2}\d{2}w)\\
	&+\left(e^{-2\lambda}\partial^2_{x_1,x_2}w-we^{-2\lambda}\s{\d{1}\n}{\d{2}\n}\right)\d{1}w
	-e^{-4\lambda}\bigg(\s{\partial^2_{x_1,x_2}\phi}{\d{1}\phi}\d{1}w+\s{\partial^2_{x_1,x_2}\phi}{\d{2}\phi}\d{2}w\bigg)\d{1}w\\
	&+w(\I_{1,2}\I_{1,1}\d{1}w+\I_{1,2}\I_{1,2}\d{2}w)\\
	&=-e^{-2\lambda}\partial^2_{x_1}w\,\d{2}w+e^{-2\lambda}\partial^2_{x_1,x_2}w\,\d{1}w-w\d{2}w\left(\I_{1,1}\I_{2,2}-\I_{1,2}^2\right)+w\d{2}w|\D_{\e_1}\n|^2\\
	&+e^{-4\lambda}\bigg(\left(\s{\partial^2_{x_1}\phi}{\d{1}\phi}\d{1}w+\s{\partial^2_{x_1}\phi}{\d{2}\phi}\d{2}w\right)\d{2}w\\
	&-\left(\s{\partial^2_{x_1,x_2}\phi}{\d{1}\phi}\d{1}w+\s{\partial^2_{x_1,x_2}\phi}{\d{2}\phi}\d{2}w\right)\d{1}w\bigg)\\
	&=-e^{-2\lambda}\partial^2_{x_1}w\,\d{2}w+e^{-2\lambda}\partial^2_{x_1,x_2}w\,\d{1}w-2K\,w\d{2}w\\
	&+e^{-4\lambda}\bigg\{\left(\s{\partial^2_{x_1}\phi}{\d{1}\phi}\d{1}w+\s{\partial^2_{x_1}\phi}{\d{2}\phi}\d{2}w\right)\d{2}w\\
	&-\left(\s{\partial^2_{x_1,x_2}\phi}{\d{1}\phi}\d{1}w+\s{\partial^2_{x_1,x_2}\phi}{\d{2}\phi}\d{2}w\right)\d{1}w\bigg\}
	\end{align*}
	
	And
	\begin{align*}
	&e^{-2\lambda}\frac{d}{dt}\left(\s{\vec{\I}_{2,2}}{\d{1}\w}-\s{\vec{\I}_{1,2}}{\d{2}\w}\right)=e^{-2\lambda}\partial^2_{x_2}w\d{1}w-w\d{1}w|\D_{\e_2}\n|^2+w(\I_{2,2}\I_{1,1}\d{1}w+\I_{2,2}\I_{1,2}\d{2}w)\\
	&-e^{-4\lambda}\left(\s{\partial^2_{x_2}\phi}{\d{1}\phi}\d{1}w+\s{\partial^2_{x_2}\phi}{\d{2}\phi}\d{2}\w\right)\d{1}w
	-e^{-2\lambda}\partial^2_{x_1,x_2}w\\
	&-w\left(\I_{1,2}\I_{1,2}\d{1}w+\I_{1,2}\I_{2,2}\d{2}w\right)
	+e^{-4\lambda}\left(\s{\partial^2_{x_1,x_2}\phi}{\d{1}\phi}\d{1}w+\s{\partial^2_{x_1,x_2}\phi}{\d{2}\phi}\d{2}w\right)\d{2}w\\
	&=e^{-2\lambda}\left(\partial^2_{x_2}w\,\d{1}w-\partial^2_{x_1,x_2}w\,\d{2}w\right)+2K_g\,w\d{1}w\\
	&+e^{-4\lambda}\bigg\{-\left(\s{\partial^2_{x_2}\phi}{\d{1}\phi}\d{1}w+\s{\partial^2_{x_2}\phi}{\d{2}\phi}\d{2}\w\right)\d{1}w\\
	&+\left(\s{\partial^2_{x_1,x_2}\phi}{\d{1}\phi}\d{1}w+\s{\partial^2_{x_1,x_2}\phi}{\d{2}\phi}\d{2}w\right)\d{2}w\bigg\}
	\end{align*}
	
	So we get
	\begin{align*}
	&\frac{d}{dt}\left(\vec{\I}\res\star d\w\right)=\bigg\{-e^{-2\lambda}\partial^2_{x_1}w\,\d{2}w+e^{-2\lambda}\partial^2_{x_1,x_2}w\,\d{1}w-2K_g\,w\d{2}w\\
	&+e^{-4\lambda}\bigg\{\left(\s{\partial^2_{x_1}\phi}{\d{1}\phi}\d{1}w+\s{\partial^2_{x_1}\phi}{\d{2}\phi}\d{2}w\right)\d{2}w\\
	&-\left(\s{\partial^2_{x_1,x_2}\phi}{\d{1}\phi}\d{1}w+\s{\partial^2_{x_1,x_2}\phi}{\d{2}\phi}\d{2}w\right)\d{1}w\bigg\}\bigg\}dx_1\\
	&+\bigg\{e^{-2\lambda}\left(\partial^2_{x_2}w\,\d{1}w-\partial^2_{x_1,x_2}w\,\d{2}w\right)+2K_g\,w\d{1}w\\
	&e^{-4\lambda}\bigg\{-\left(\s{\partial^2_{x_2}\phi}{\d{1}\phi}\d{1}w+\s{\partial^2_{x_2}\phi}{\d{2}\phi}\d{2}\w\right)\d{1}w\\
	&+\left(\s{\partial^2_{x_1,x_2}\phi}{\d{1}\phi}\d{1}w+\s{\partial^2_{x_1,x_2}\phi}{\d{2}\phi}\d{2}w\right)\d{2}w\bigg\}\bigg\}dx_2	
	\end{align*}
	Now, as $M$ is minimal, $\phi:\Sigma\rightarrow\R^3$ is harmonic for the metric $g$, so in our conformal chart,
	\begin{align*}
	\Delta_g\phi=e^{-2\lambda}(\partial^2_{x_1}\phi+\partial^2_{x_2}\phi)=0
	\end{align*}
	thus $\partial^2_{x_2}\phi=-\partial^2_{x_1}\phi$. And by conformity, 
	\begin{align*}
	\s{\d{i}\phi}{\d{j}\phi}=e^{2\lambda}\delta_{i,j},\quad \s{\partial^2_{x_i,x_j}\phi}{\d{j}\phi}=\frac{1}{2}\d{i}e^{2\lambda},\quad 1\leq i,j\leq 2.
	\end{align*}
	We deduce that
	\begin{align*}
	&\left(\s{\partial^2_{x_1}\phi}{\d{1}\phi}\d{1}w+\s{\partial^2_{x_1}\phi}{\d{2}\phi}\d{2}w\right)\d{2}w
	-\left(\s{\partial^2_{x_1,x_2}\phi}{\d{1}\phi}\d{1}w+\s{\partial^2_{x_1,x_2}\phi}{\d{2}\phi}\d{2}w\right)\d{1}w\\
	&=\s{\d{1}^2\phi}{\d{1}\phi}(\d{1}w)(\d{2}w)-\s{\d{2}^2\phi}{\d{2}\phi}(\d{2}w)^2
	-\s{\partial^2_{x_1,x_2}\phi}{\d{1}\phi}(\d{1}w)^2\\
	&-\s{\partial^2_{x_1,x_2}\phi}{\d{2}\phi}(\d{1}w)(\d{2}w)\\
	&=\frac{1}{2}\bigg(\d{1}e^{2\lambda}(\d{1}w)(\d{2}w)-\d{2}e^{2\lambda}(\d{2}w)^2-\d{2}e^{2\lambda}(\d{1}w)^2-\d{1}e^{2\lambda}(\d{1}w)(\d{2}w)\bigg)\\
	&=-\frac{1}{2}\d{2}e^{2\lambda}\left(\d{1}w)^2+(\d{2}w)^2\right)
	\end{align*}
	And likewise,
	\begin{align*}
	&\left(\s{\partial^2_{x_1,x_2}\phi}{\d{1}\phi}\d{1}w+\s{\partial^2_{x_1,x_2}\phi}{\d{2}\phi}\d{2}w\right)\d{2}w+\left(\s{\partial^2_{x_1}\phi}{\d{1}\phi}\d{1}w-\s{\partial^2_{x_2}\phi}{\d{2}\phi}\d{2}\w\right)\d{1}w\\
	&=\frac{1}{2}\d{1}e^{2\lambda}\left((\d{1}w)^2+(\d{2}w)^2\right)
	\end{align*}
	so
	\begin{align*}
	&e^{-4\lambda}\bigg\{\left(\s{\partial^2_{x_1}\phi}{\d{1}\phi}\d{1}w+\s{\partial^2_{x_1}\phi}{\d{2}\phi}\d{2}w\right)\d{2}w
	-\Big(\s{\partial^2_{x_1,x_2}\phi}{\d{1}\phi}\d{1}w\\
	&+\s{\partial^2_{x_1,x_2}\phi}{\d{2}\phi}\d{2}w\Big)\d{1}w\bigg\}dx_1
	+e^{-4\lambda}\bigg\{-\left(\s{\partial^2_{x_2}\phi}{\d{1}\phi}\d{1}w+\s{\partial^2_{x_2}\phi}{\d{2}\phi}\d{2}\w\right)\d{1}w\\
	&+\left(\s{\partial^2_{x_1,x_2}\phi}{\d{1}\phi}\d{1}w+\s{\partial^2_{x_1,x_2}\phi}{\d{2}\phi}\d{2}w\right)\d{2}w\bigg\}dx_2\\
	&=\frac{1}{2}e^{-4\lambda}|dw|^2\left(\d{1}e^{2\lambda}dx_2-\d{2}e^{2\lambda}dx_1\right)\\
	&=\frac{1}{2}e^{-4\lambda}|dw|^2\star de^{2\lambda}=\frac{1}{2}(e^{-2\lambda}\star de^{2\lambda})|dw|_g^2
	\end{align*}
	The remaining terms are
	\begin{align*}
	&e^{-2\lambda}\left\{\left(-\partial^2_{x_1}w\,\d{2}w+\partial^2_{x_1,x_2}w\,\d{1}w\right)dx_1+\left(\partial^2_{x_2}w\,\d{1}w-\partial^2_{x_1,x_2}w\,\d{2}w\right)dx_2\right\}=\\
	&=e^{-2\lambda}\left(\partial^2_{x_1}w(-\d{2}wdx_1)+\partial_{x_2}^2w(\d{1}wdx_2)\right)+\frac{1}{2}e^{-2\lambda}\left(\d{2}|\d{1}w|^2dx_1-\d{1}|\d{2}w|^2dx_2\right)\\
	&=e^{-2\lambda}\left(\partial^2_{x_1}w(\star dw-\d{1}wdx_2)+\partial_{x_2}^2w(\star dw+\d{2}wdx_1)\right)+\frac{1}{2}e^{-2\lambda}\left(\d{2}|\d{1}w|^2dx_1-\d{1}|\d{2}w|^2dx_2\right)\\
	&=\Delta_g\, w(\star dw)+\frac{1}{2}e^{-2\lambda}\left(\d{2}(|\d{1}w|^2+|\d{2}w|^2)dx_1-\d{1}(|\d{1}w|^2+|\d{2}w|^2)dx_2\right)\\
	&=\Delta_g w\,\star dw-\frac{1}{2}e^{-2\lambda}\star d|dw|^2
	\end{align*}
	and
	\begin{align*}
	2Kw(-\d{2}wdx_1+\d{1}wdx_2)=2Kw\star dw
	\end{align*}
	And we have
	\begin{align*}
	e^{-2\lambda}\star d |dw|^2=e^{-2\lambda}\star d(e^{2\lambda}|dw|_g^2)=(e^{-2\lambda}\star de^{2\lambda})|dw|_g^2+\star d|dw|_g^2
	\end{align*}
	so we deduce that
	\begin{align*}
	\frac{d}{dt}\left(\vec{\I}\res \star dw\right)_{|t=0}&=\frac{1}{2}(e^{-2\lambda}\star de^{2\lambda})|dw|_g^2+\Delta_gw \star dw-\frac{1}{2}\left(\left(e^{-2\lambda}\star de^{2\lambda}\right)|dw|_g^2+\star d|dw|_g^2\right)+2Kw\star dw\\
	&=(\Delta_gw+2Kw)\star dw-\frac{1}{2}\star d|dw|_g^2.
	\end{align*}
	This concludes the proof of the lemma.
\end{proof}

\section{Index estimate}\label{proofmain}

\subsection{Introduction and definitions}

According to the classification established by Robert Bryant in \cite{bryant}, we can relate the theory of Willmore surfaces with the theory of minimal surfaces. We will give first some precisions about the latest notion.
Let $\phi:\bar{\Sigma}\rightarrow\R^3$ a minimal isometric immersion from an orientable Riemann surface $\bar{\Sigma}$. By minimal we mean that the mean curvature tensor $\vec{H}_g$ (where $g$ is the pull-back by $\phi$ of the euclidean metric on $\bar{\Sigma}$) of $\phi$ is identically zero. We say that $M=\phi(\bar{\Sigma})$ has finite total curvature if
\[
C(M)=\int_{\bar{\Sigma}}|K_g|d\vg=\int_{\bar{\Sigma}}-K_gd\vg<\infty.
\]
We first recall the following theorem.
\begin{theorem}\label{rappel}
	\begin{itemize}
		\item[1)] \textnormal{(\cite{huber})} $\bar{\Sigma}$ is conformally diffeomorphic to a compact Riemann surface $\Sigma$ with a finite number a points removed, called the ends of $\bar{\Sigma}$. We say that an end $p\in \Sigma$ is embedded, if there exists a radius $r>0$, such that $\phi|D^2(p,r)\setminus \ens{p}$ is an embedding.
		\item[2)] \textnormal{(\cite{schoenPlanar})} If $p\in\Sigma$ is an embedded end, and $r>0$ is such that $\phi|D^2(p,r)\setminus \ens{p}$ is an embedding, then there exists $a,b\in\R$ and $c\in\R^2$ such that, up to rotation and translation
		\begin{align*}
			\phi(D^2(p,r)\setminus\ens{p})=\ens{(x,y):y=a\log|x|+b+\frac{c\cdot x}{|x|^2}+O\left(\frac{1}{|x|^2}\right),\;x\in\R^2\setminus D^2(0,1)}.
		\end{align*}
		We say that the end is planar (or has zero logarithmic growth) if $a=0$.
		\item[3)] \textnormal{(\cite{osserman})} $C(M)$ is an integer multiple of $2\pi$, and furthermore, if all ends of $\bar{\Sigma}$ are embedded, and $\bar{\Sigma}$ has $m$ ends, then
		\begin{equation}
			C(M)=\int_{\Sigma}-K_gd\vg=4\pi\left(m+\gamma-1\right).
		\end{equation} 
		if the genus of $\Sigma$ is $\gamma$.
	\end{itemize}
\end{theorem}
We at present able to state the fundamental theorem of Robert Bryant of interest here.
\begin{theorem}\label{theobryant}[Bryant] Let $\psi:S^2\rightarrow \R^3$ a Willmore immersion, then either $\psi$ is totally umbilic, either $\psi$ is the inversion of a minimal surface with finite total curvature and planar ends.
\end{theorem}

\subsection{Second variation of conformal Willmore functional}

We first define the index for the Willmore functional.
\begin{defi}
	Let $\psi\in\im\cap C^{\infty}(\Sigma,N^n))$ a Willmore surface. Then the second variation $D^2W$ of $W=W_{N^n}$ is well defined by 
	\begin{align*}
	D^2W(\psi)[\w,\w]=\frac{d^2}{dt^2}W(\psi_t)_{|t=0}.
	\end{align*}
	for all $\{\psi_t\}_{t\in I}\in C^2(I,\im\cap C^{\infty}(\Sigma,N^n))$ is a $C^2$ family of immersions such that $\psi_0=\psi$, and $\w=\frac{d}{dt}(\psi_t)_{|t=0}$.
	The index of a critical point $\psi\in\im\cap\mathrm{W}^{1,\infty}(\Sigma,N^n)$ of the Willmore functional $W$, noted $\mathrm{Ind}_W(\psi)$, is defined as the dimension of the subspace of $\tim{\psi}\cap\mathrm{W}^{1,\infty}(\Sigma,TN^n)$ where the second derivative $D^2W(\psi)$ is negative definite. We define in analogous way the $\W$ index, noted $\mathrm{Ind}_{\W}(\psi)$.
\end{defi}

 Let $\Sigma$ a closed Riemann surface and $\psi:\Sigma\rightarrow\R^3$ a non-banched Willmore immersion which is the inversion of a minimal surface with planar ends. Then up to translation there exists an integer $m$, $\ens{p_1,\cdots,p_m}\subset \Sigma$ a finite set of $m$ distinct points, and a minimal isometric immersion $\phi:\Sigma\rightarrow\R^3$, such that $M=\phi(\bar{\Sigma})$ ($\bar{\Sigma}={\Sigma}\setminus\ens{p_1,\cdots,p_m}$) is a complete minimal surface with finite total curvature, and planar ends, and if $i:\R^3\setminus\ens{0}\rightarrow\R^3\setminus\ens{0}$ is the inversion at $0$, given by
\begin{align*}
	i(x)=\frac{x}{|x|^2},\quad \forall x\in\R^3\setminus\ens{0},
\end{align*}
then we have $\psi=i\circ\phi$. By conformal invariance, noting $\W=\W_{\R^3}$, we have
\begin{align*}
	\W(\phi)=\W(\psi).
\end{align*}
So by theorem \eqref{rappel}, we have
\begin{align*}
	W(\psi)=\W(\psi)+2\pi\chi(\Sigma)=C(M)+2\pi\chi(\Sigma)=4\pi m,
\end{align*}
which shows that Willmore energy is quantized by $4\pi$. Another interpretation of this integer $m$ is furnished by a theorem of Peter Li and Shing-Tung Yau \cite{lieyau}. They proved that if $x\in \R^3$, and $\psi^{-1}(\ens{x})=k\in\N$, then
\begin{align*}
	W(\psi)\geq 4\pi k
\end{align*}
so $m$ is the maximum number if preimages under $\psi$ of points in $\R^3$. Furthermore, thanks our normalisation, we see that $0$ has $m$ premiages by $\psi$.

Now, as $M$ is minimal, $\phi$ is a Willmore immersion. 
Now, we remark that for all admissible variation $\{\psi\}_{t\in I}$, we have
\begin{align*}
\mathscr{W}(\psi_t)=\int_{S^2}(H^2_{g_t}-K_{g_t})d\mathrm{vol}_{g_t}=\int_{S^2}H^2_{g_t}d\mathrm{vol}_{g_t}-2\pi\chi(\Sigma)
\end{align*}
so, by conformal invariance of $\mathscr{W}$, for a normal variation $\vec{v}=v\n$ of $\psi$, we have
\begin{align*}
	D^2W(\psi)[\vec{v},\vec{v}]=D^2\W(\psi)[\vec{v},\vec{v}]=\frac{d^2}{dt^2}\W(\psi_t)=\frac{d^2}{dt^2}\W(\iota\circ \psi_t)=D^2\W(\phi)[\w,\w].
\end{align*}
if $\w=|\phi|^2\vec{v}$.
Indeed, we have
\begin{align*}
	\frac{d}{dt}{\left(\psi_t\right)}_{|t=0}=\vec{v}=v\n,
\end{align*}
so (as $\vec{v}$ is normal),
\begin{align*}
	\frac{d}{dt}{\left(\iota\circ\psi_t\right)}_{|t=0}=\frac{\vec{v}}{|\psi|^2}-2\frac{\phi\cdot \vec{v}}{|\psi|^4}=\frac{\vec{v}}{|\psi|^2}=|\phi|^2\vec{v}.
\end{align*}

By \eqref{d2minfacile} and lemma \ref{d2k}, we have the following formula.

\begin{theorem}\label{der2}
	Let $\Sigma$ a Riemann surface, an $\phi:\Sigma\rightarrow \R^3$ a smooth minimal immersion. Then for all smooth normal variation $\w=w\n$, we have
	\begin{equation}
	D^2\mathscr{W}(\phi)[\w,\w]=\int_{\Sigma}\left\{\frac{1}{2}(\Delta_gw-2Kw)^2d\vg-d\left(\left(\Delta_gw+2Kw\right)\star dw-\frac{1}{2}\star d |dw|_g^2\right)\right\}.\label{formula}
	\end{equation}
\end{theorem}
 By the previous discussion, $\mathrm{Ind}_{\W}(\psi)$ is equal to the index of $\phi$ for normal variations of the form $\w=|\phi|^2v\n$. A remarkable fact is that one can estimate the index by computing explicitly the integral involving the residue term.

\subsection{Explicit formula for the second derivative of inversions of minimal surfaces}

We can state the main result of this section, from which theorem \ref{main} will be an easy consequence .

\begin{theorem}
	Let $\Sigma$ a closed Riemann surface and $\phi:\Sigma\setminus\ens{p_1,\cdots , p_m}\rightarrow \R^3$ a minimal immersion with $m$ embedded planar ends, such that $\psi=i\circ \phi:\Sigma\rightarrow \R^3$ is a non-branched Willmore immersion. Then for all normal variation $\vec{v}=v\n$ of $\psi$, we have
	\begin{equation}\label{magique}
	D^2\mathscr{W}(\psi)[\vec{v},\vec{v}]=\lim_{R\rightarrow 0}\frac{1}{2}\int_{\Sigma_R}(\Delta_g w-2K_gw)^2d\vg-4\pi\sum_{j=1}^m\frac{\mathrm{Res}_{p_j}\phi}{R^2}\,v(p_j)^2
	\end{equation}
	if $w=|\phi|^2v$, and $g=\phi^{\ast}g_{\,\R^3}$, and
	$
	\displaystyle\Sigma_R=\Sigma\setminus\bigcup_{j=1}^{m}D_{\Sigma}^2(p_j,R).
	$
\end{theorem}
\begin{proof}
Let $p\in \Sigma$ an end of $M=\phi(\Sigma)$.
Thanks of a theorem of Richard Schoen (\cite{schoenPlanar} see also the paper of Robert Osserman \cite{osserman}) about the Weierstrass-Enneper parametrisation, as the end is embedded and planar by Bryant's theorem, there exists a radius $R_0>0$, and a local chart $D^2_{\R^2}(0,1)\rightarrow D^2_{S^2}(p,R_0)$, which sends $0$ to $p$ such that in this chart, we have
\begin{align*}
	\phi(u)=\Re\int_{}^{z}(\varphi_1,\varphi_2,\varphi_3)d\zeta
\end{align*}
where $z=u_1+iu_2\in D^2\setminus \ens{0}$, and for $1\leq j\leq 3$, $\varphi_j$ is an holomorphic function with a pole of order at most $2$ at $0$, and 
\begin{equation}\label{sumsquare}
\varphi_1^2+\varphi_2^2+\varphi_3^2=0
\end{equation}
Assuming up to rotation that the asymptotic normal is $(0,0,1)$, this translates to
\begin{align*}
	\varphi_1(z)=\frac{a_1}{z^2}+O(1),\quad\varphi_2(z)=\frac{a_2}{z^2}+O(1),\quad \varphi_3(z)=O(1)
\end{align*}
where $a_1,a_2\in \C\setminus\ens{0}$, and condition \eqref{sumsquare} shows that $a_1^2+a_2^2=0$. Therefore, if we define $f$ by
\begin{align*}
	f(z)=\int^{z}(\varphi_1,\varphi_2,\varphi_3)d\zeta
\end{align*}
we have
\begin{align*}
	f(z)=\left(-\frac{a_1}{z}+O(1),-\frac{a_2}{z}+O(1),O(|z|)\right)
\end{align*}
and 
\begin{align*}
	\mathrm{Res}_{\,0}f(z)dz=\left(-a_1,-a_2,0\right)
\end{align*}
and as $a_1^2+a_2^2=0$, we have
\begin{align*}
	|\mathrm{Res}_{\,0} f(z)dz|^2=2|a_1|^2=2|a_2|^2
\end{align*}
so assuming up to a change of coordinate that $a_1=-\alpha\in\R\setminus\ens{0}$, $a_2=-i\alpha$, and
\begin{align*}
\phi(u)=\left(\alpha\frac{u}{|u|^2}+O(|u|), \beta\cdot u+O(|u|^2)\right), \quad \forall u\in D^2\setminus\ens{0},
\end{align*}
where $\beta\in\R^2$. So the quantity
\begin{equation}\label{defresidu}
\mathrm{Res}_p(\phi)=2\alpha^2=|\mathrm{Res}_{\,0} f(z)dz|^2=2\lim_{|u|\rightarrow 0}|u|^2|\phi(u)|^2
\end{equation}
is well-defined and we call it the residue of $\phi$ at $p$.
As $R$ goes to $0$, for asymptotic formulas, we can assume that for all $0<R\leq R_0$, the chart sends the circle $S^1(0,R)\subset \R^2$ to the circle $S^1(p,R)$.
In the rest of the proof, we assume that the residue is equal to $1$, and we shall see that it will only appear as a multiplicative factor. Now, making a change of variable, we get new coordinates writing
\begin{align*}
	(x_1,x_2,x_3)=\phi(u)
\end{align*}
we get
\begin{align}\label{changevariable}
	u=\frac{1}{\alpha}\frac{x}{|x|^2}+O\left(\frac{1}{|x|^3}\right),
\end{align}
so in $(x_1,x_2)$ coordinates, we get
\begin{align*}
	\phi(x)=\left(x,a+\frac{b\cdot x}{|x|^2}+O\left(\frac{1}{|x|^2}\right)\right),\quad a\in\R,\;b\in\R^2
\end{align*}
where $x\in \R^2\setminus K$, where $K$ is a compact set containing $0$. The components of the induced metric $g$ in these coordinates are
\begin{align*}
	g_{i,j}(x)=\s{\d{i}\phi(x)}{\d{j}\phi(x)},\quad i,j=1,2.
\end{align*}
and
\begin{align*}
	\d{1}\phi(x)&=\left(1,0,\frac{b_1}{|x|^2}-2\frac{(b\cdot x)x_1}{|x|^4}+O\left(\frac{1}{|x|^3}\right)\right)\\
	\d{2}\phi(x)&=\left(0,1,\frac{b_2}{|x|^2}-2\frac{(b\cdot x)x_2}{|x|^4}+O\left(\frac{1}{|x|^3}\right)\right).
\end{align*}

We can differentiate under the $O$ sign, as the Weierstrass-Enneper parametrization shows that $\phi$ is locally the real part of a meromorphic function, therefore is it analytic outside of branch points, and the rest can be differentiated as the rest of a convergent series of power of $|x|^{-1}$.
These expressions show that we have conformity and a flat metric at infinity, as
\begin{align*}
	g_{1,2}(x)&=\frac{b_1b_2}{|x|^4}-2(b_1x_2+b_2x_1)\frac{b\cdot x}{|x|^6}+4x_1x_2\frac{(b\cdot x)^2}{|x|^8}+O\left(\frac{1}{|x|^5}\right)=O\left(\frac{1}{|x|^4}\right)\\
	&g_{i,i}(x)=1+\frac{\left(b_i|x|^2-2x_i(b\cdot x)\right)^2}{|x|^8}+O\left(\frac{1}{|x|^5}\right)=1+O\left(\frac{1}{|x|^4}\right),\quad i=1,2
\end{align*}
and
\begin{align*}
	\det g(x)&=1+\sum_{i=1}^2\frac{\left(b_i|x|^2-2x_i(b\cdot x)\right)^2}{|x|^8}+\frac{1}{|x|^{16}}(b_1|x|^2-2x_1(b\cdot x))^2(b_2|x|^2-2x_2(b\cdot x))^2+O\left(\frac{1}{|x|^5}\right)\\
	&-\left(\frac{b_1b_2}{|x|^4}-2(b_1x_2+b_2x_1)\frac{b\cdot x}{|x|^6}+4x_1x_2\frac{(b\cdot x)^2}{|x|^8}\right)^2+O\left(\frac{1}{|x|^9}\right)\\
	&=1+\sum_{i=1}^2\frac{\left(b_i|x|^2-2x_i(b\cdot x)\right)^2}{|x|^8}+O\left(\frac{1}{|x|^5}\right)=1+O\left(\frac{1}{|x|^4}\right).
\end{align*}
So we have
\begin{align}\label{flatmetric}
	g_{i,j}&=\delta_{i,j}+O\left(\frac{1}{|x|^4}\right)\nonumber\\
	d\vg(x)&=\sqrt{\det(g(x))}dx_1\wedge dx_2=\left(1+O\left(\frac{1}{|x|^4}\right)\right)dx_1\wedge dx_2
\end{align}
which proves asymptotic flatness.

Now, we have by Stokes theorem
\begin{equation}\label{stokes}
\int_{\Sigma_R}d\left((\Delta_gw+2Kw)\star dw-\frac{1}{2}\star d|dw|_g^2\right)=-\sum_{j=1}^m\int_{S^1_{\Sigma}(p_j,R)}\left((\Delta_gw+2Kw)\star dw-\frac{1}{2}\star d|dw|_g^2\right)
\end{equation}
where the circles $S^1_{\Sigma}(p_i,R)$ are \textit{positively} oriented (which explains the negative sign in front of the sum). Setting $r=\alpha (R)^{-1}$, the change of variable \eqref{changevariable} shows that
\begin{align*}
	\int_{S^1_{\Sigma}(p_j,R)}\left((\Delta_gw+2Kw)\star dw-\frac{1}{2}\star d|dw|_g^2\right)=\int_{S^1_{r}}\left((\Delta_gw+2Kw)\star dw-\frac{1}{2}\star d|dw|_g^2\right)+o(1)
\end{align*}
where $S^1_r$ is the circle in $\R^2$ of radius $r>0$. Indeed, if we make the change of variable 
\begin{align*}
	u=\frac{1}{\alpha}\frac{x}{|x|^2}
\end{align*}
then
\begin{align*}
	\phi(x)=\left(x+O\left(\frac{1}{|x|}\right),\tilde{a}+\frac{\tilde{b}\cdot x}{|x|^2}+O\left(\frac{1}{|x|^2}\right)\right)
\end{align*}
for some $\tilde{a}\in\R$, $\tilde{b}\in\R^2$. As we will see in the following, the error term in these coordinates where $\phi$ is almost a graph is irrelevant, so we discard it and use \eqref{changevariable} instead. Now thanks of \eqref{changevariable}
\begin{align*}
	w(x)=\frac{v\left(\frac{x}{|x|^2}\right)}{|\vec{\Psi}(x)|^2}=v\left(\frac{x}{|x|^2}\right)|\phi(x)|^2=v\left(\frac{x}{|x|^2}\right)\left(|x|^2+a^2+O\left(\frac{1}{|x|^2}\right)\right)
\end{align*}

We now have
\begin{align*}
    &\star dw(x)=\d{1}\left(\left(|x|^2+a^2+O(|x|^{-2})\right)v\left(\frac{x}{|x|^2}\right)\right)dx_2-\d{2}\left(\left(|x|^2+a^2+O(|x|^{-2})\right)v\left(\frac{x}{|x|^2}\right)\right)dx_1\\
	&=2v(i(x))\left(x_1dx_2-x_2dx_1\right)+O(|x|^{-3})v(i(x))(dx_2-dx_1)\\
	&+(|x|^2+a^2+O(|x|^{-2}))\bigg\{\left(\left(\frac{1}{|x|^2}-\frac{2x_1^2}{|x|^4}\right)\partial_1 v(i(x))-\frac{2x_1x_2}{|x|^4}\partial_2v(i(x))\right)dx_2\\
	&-\left(\left(\frac{1}{|x|^2}-\frac{2x_2^2}{|x|^4}\right)\partial_2 v(i(x))-\frac{2x_1x_2}{|x|^4}\partial_1v(i(x))\right)dx_1\bigg\}\\
	&=2v(i(x))(x_1dx_2-x_2dx_1)+\left(\frac{x_2^2-x_1^2}{|x|^2}\p{1}v-\frac{2x_1x_2}{|x|^2}\p{2}v\right)dx_2-\left(\frac{x_1^2-x_2^2}{|x|^2}\p{2}-\frac{2x_1x_2}{|x|^4}\p{1}v\right)dx_1+O\left(\frac{1}{|x|^2}\right)
\end{align*}

Now, by \eqref{flatmetric},
\begin{align*}
	\Delta_gw(x)&=\frac{1}{\sqrt{g(x)}}\sum_{i,j=1}^2\d{i}\left(\sqrt{g(x)}g^{i,j}(x)\d{j}\left(|\phi(x)|^2v(i(x))\right)\right)\\
	&=\frac{1}{\sqrt{g(x)}}\bigg(\d{1}\left((1+O(|x|^{-4})\d{1}((|x|^2+a^2+O(|x|^{-2}))v(i(x))))\right)\\
	&+\d{2}\left((1+O(|x|^{-4})\d{2}((|x|^2+a^2+O(|x|^{-2}))v(i(x))))\right)\\
	&+\d{1}\left(O(|x|^{-4})\d{2}((|x|^2+a^2+O(|x|^{-2}))v(i(x)))\right)\\
	&+\d{2}\left(O(|x|^{-4})\d{1}((|x|^2+a^2+O(|x|^{-2}))v(i(x)))\right)\bigg)
\end{align*}
Now, as $\star dw(x)=O(|x|)$, all terms in $\Delta_gw$ of lower order than $|x|^{-2}$ will vanish as $R\rightarrow\infty$, as
\begin{align*}
	\int_{S^1_r}\frac{1}{|x|^2}d\h^1(x)=\frac{2\pi}{r}\conv[r]0.
\end{align*}
Now, we note that
\begin{align*}
	v(i(x))=O(1),\quad dv(\iota(x))=O(|x|^{-2})
\end{align*}
so
\begin{align*}
	\d{i}\left(O(|x|^{-4})\d{j}\left((|x|^2+a^2+O(|x|^{-2})v(i(x))\right)\right)&=\d{i}\left(O(|x|^{-4})\left(2x_jv(i(x))+(|x|^2+a^2+|x|^{-2})\d{j}v(i(x))\right)\right)\\
	&=\d{i}(O(|x|^{-4})(O(|x|+O(1))))=O(|x|^{-4}),
\end{align*}
so only the flat Laplacian will remain in the end. We have
\begin{align*}
	\Delta \left(|x|^2+a^2+O(|x|^{-2})v(\iota(x))\right)&=(4+O(|x|^{-4}))v(\iota(x))+4\s{x+O(|x|^{-3})}{\D v(\iota(x))}\\
	&+(|x|^2+a^2+O(|x|^{-2}))\Delta v(\iota(x)).
\end{align*}
A simple computation will show that $\Delta v(\iota(x))=O(|x|^{-3})$, so we can discard all error terms, and the constant term $a^2$. So it suffices to compute the flat laplacian of $w$. We have
\begin{align*}
	\Delta w(i(x))=4v(i(x))+2\left(2x_1\d{1}v(i(x))+2x_2\d{2}v(i(x))\right)+|x|^2\Delta v(i(x))
\end{align*}
and
\begin{align*}
	x_1\d{1}v(i(x))+x_2\d{2}v(i(x))&=x_1\left(\frac{x_2^2-x_1^2}{|x|^4}\p{1}v-\frac{2x_1x_2}{|x|^4}\p{2}v\right)+x_2\left(\frac{x_1^2-x_2^2}{|x|^4}\p{2}v-\frac{2x_1x_2}{|x|^4}\p{1}v\right)\\
	&=\frac{1}{|x|^4}\left(x_1(x_2^2-x_1^2)-2x_1x_2^2\right)\p{1}+\frac{1}{|x|^4}\left(-2x_1^2x_2+x_2(x_1^2-x_2^2)\right)\\
	&=-\frac{1}{|x|^2}\left(x_1\p{1}v+x_2\p{2}v\right)\\
	&=-\s{dv(i(x))}{i(x))}
\end{align*}

We now compute (omitting the argument  $i(x)$)
\begin{align*}
\Delta v(i(x))&=\d{1}\left(\frac{x_2^2-x_1^2}{|x|^4}\p{1}v(i(x))-\frac{2x_1x_2}{|x|^4}\p{2}v(i(x))\right)+\d{2}\left(\frac{x_1^2-x_2^2}{|x|^4}\p{2}v(i(x))-\frac{2x_1x_2}{|x|^4}\p{1}v(i(x))\right)\\
&=\left(-\frac{2x_1}{|x|^4}-\frac{4x_1(x_2^2-x_1^2)}{|x|^6}\right)\p{1}v+\left(-\frac{2x_2}{|x|^4}+\frac{8x_1^2x_2}{|x|^6}\right)\p{2}v+\frac{x_2^2-x_1^2}{|x|^4}\left(\frac{x_2^2-x_1^2}{|x|^4}\p{1}^2v-\frac{2x_1x_2}{|x|^4}\p{12}^2v\right)\\
&-\frac{2x_1x_2}{|x|^4}\left(\frac{x_2^2-x_1^2}{|x|^4}\p{12}^2v-\frac{2x_1x_2}{|x|^4}\p{2}^2v\right)+
\left(-\frac{2x_2}{|x|^4}-\frac{4x_2(x_1^2-x_2^2)}{|x|^6}\right)\p{2}v+\left(-\frac{2x_1}{|x|^4}+\frac{8x_1x_2^2}{|x|^6}\right)\p{1}v\\&+\frac{x_1^2-x_2^2}{|x|^4}\left(\frac{x_1^2-x_2^2}{|x|^4}\p{2}^2v-\frac{2x_1x_2}{|x|^4}\p{12}^2v\right)
-\frac{2x_1x_2}{|x|^4}\left(\frac{x_1^2-x_2^2}{|x|^4}\p{12}^2v-\frac{2x_1x_2}{|x|^4}\p{1}^2v\right)\\
&=\frac{1}{|x|^4}\Delta v
\end{align*}
Indeed, the terms with $\p{1}v$ (resp. $\p{2}v$), is zero, as the sum of the coefficients is :
\begin{align*}
\left(-\frac{2x_1}{|x|^4}-\frac{4x_1(x_2^2-x_1^2)}{|x|^6}\right)+\left(-\frac{2x_1}{|x|^4}+\frac{8x_1x_2^2}{|x|^6}\right)=\frac{1}{|x|^6}\left(-4x_1(x_1^2+x_2^2)-4x_1(x_2^2-x_1^2)+8x_1x_2^2\right)=0
\end{align*}
The coefficient of $\p{12}^2v$ is
\begin{align*}
\frac{1}{|x|^8}\left((x_2^2-x_1^2)(-2x_1x_2)-2x_1x_2(x_2^2-x_1^2)+(x_1^2-x_2^2)(-2x_1x_2)-2x_1x_2(x_1^2-x_2^2)\right)=0
\end{align*}
And the coefficient in front of $\p{i}^2v$ (for $i=1,2$), is
\begin{align*}
\frac{1}{|x|^8}\left((x_2^2-x_1^2)^2+4x_1^2x_2^2\right)=\frac{1}{|x|^4}.
\end{align*}
So we get
\begin{equation}\label{flatdelta}
	\Delta_gw(i(x))=4v(i(x))-4\s{dv(i(x))}{i(x)}+\frac{1}{|x|^2}\Delta v(i(x))+O\left(\frac{1}{|x|^3}\right)
\end{equation}
Furthermore, we have
\begin{align*}
	\star\d{1}\phi(x)\wedge\d{2}\phi(x)=\left(O\left(\frac{1}{|x|^2}\right),O\left(\frac{1}{|x|^2}\right),1\right)
\end{align*}
so
\begin{align*}
	\n(x)=\left(O\left(\frac{1}{|x|^2}\right),O\left(\frac{1}{|x|^2}\right),1\right)
\end{align*}
and for $1\leq i,j\leq 2$,
\begin{align*}
	\partial_{x_i,x_j}\phi(x)=\left(0,0,O\left(\frac{1}{|x|^3}\right)\right)
\end{align*}
so
\begin{align*}
	\I_{i,j}=\s{\partial_{x_1,x_2}^2\phi}{\n}=O\left(\frac{1}{|x|^3}\right)
\end{align*}
and
\begin{align*}
	K(x)=O\left(\frac{1}{|x|^6}\right)
\end{align*}
so as $\star dw(i(x))=O(|x|^2)$
\begin{align}\label{nullcurvature}
	2K(x)w(i(x))\star dw(x)=2K(x)|x|^2v(i(x))O(|x|^2)=O\left(\frac{1}{|x|^6}\right)|x|^2v(i(x))O(|x|^2)=O\left(\frac{1}{|x|^2}\right)
\end{align}
so
\begin{align*}
	\int_{S^1_r}2K(x)w(i(x))\star dw(i(x))=O\left(\frac{1}{r}\right).
\end{align*}

Now, the error terms of order less than $O(|x|^{-2})$ in $|dw|^2_g$ will vanish when $r\rightarrow\infty$. By \eqref{flatmetric}, as $Dw(x)=O(|x|)$, for $i\neq j$,
\begin{align*}
	g^{i,j}(x)\d{i}w(x)\d{j}w(x)=O\left(\frac{1}{|x|^4}\right)O(|x|^2)=O\left(\frac{1}{|x|^2}\right)
\end{align*}
Therefore, we will omit all these error terms, and we do not write the $O(|x|^{-2})$ in equalities
\begin{align}\label{normedw}
	|dw(x)|_g^2&=|\d{1}w(x)|^2+|\d{2}w(x)|^2\nonumber\\
	&={\left\{\left(2x_1+O(|x|^{-3})\right)v(i(x))+(|x|^{-2}+a^2|x|^{-4}+O(|x|^{-6}))\left((x_2^2-x_1^2)\partial_1v(i(x))-2x_1x_2\partial_2v(i(x))\right)\right\}}^2\nonumber\\
	&+{\left\{\left(2x_2+O(|x|^{-3})\right)v(i(x))+(|x|^{-2}+a^2|x|^{-4}+O(|x|^{-6}))\left((x_2^2-x_1^2)\partial_2v(i(x))-2x_1x_2\partial_1v(i(x))\right)\right\}}^2\nonumber\\
	&=\left(2x_1v(i(x))+|x|^{-2}((x_2^2-x_1^2)\partial_1v(i(x))-2x_1x_2\partial_2v(i(x))\right)^2\nonumber\\
	&+\left(2x_2v(i(x))+|x|^{-2}((x_1^2-x_2^2)\partial_2v(i(x))-2x_1x_2\partial_2v_1(i(x))\right)^2\nonumber\\
	&=4|x|^2v(i(x))^2+(\partial_1v(i(x)))^2+(\partial_2v(i(x)))^2-4v(i(x))\s{x}{Dv(i(x))}
\end{align}
Now, $D^kv(i(x))=O(r^{-2k})$ for all $k\geq 0$, so
\begin{align*}
\star d \left((\partial_1v(i(x)))^2+(\partial_2v(i(x)))^2\right)=O\left(\frac{1}{|x|^2}\right),
\end{align*}
so we can drop these terms in the $O(|x|^{-2})$ error.
So we have by \eqref{flatdelta}, \eqref{nullcurvature}, \eqref{normedw}
\begin{align*}
	\Delta_gw(i(x))-\frac{1}{2}\star d |dw|_g^2&=\left(4v(i(x))-4\s{dv(i(x))}{i(x)}+\frac{1}{|x|^2}\Delta v(i(x))+O\left(\frac{1}{|x|^3}\right)\right)\star d\left(|x|^2v(i(x))\right)\\
	&-\frac{1}{2}\star d\left(4|x|^2v(i(x))^2-4v(i(x))\s{dv(i(x))}{x}\right)+O\left(\frac{1}{|x|^2}\right)
\end{align*}
And we remark that
\begin{align*}
	\star d\left(|x|^2v(i(x))^2\right)=2|x|^2v(i(x))\star d\left(v(i(x))\right)+v(i(x))^2\star d\left(|x|^2 \right)
\end{align*}
so
\begin{align*}
	4v(i(x))\star d\left(|x|^2v(i(x))\right)-\frac{1}{2}\star d\left(4|x|^2v(i(x))\right)=2v(i(x))^2\star d(|x|^2).
\end{align*}
Then, we have
\begin{align*}
	\star d\left(v(i(x))\s{dv(i(x))}{x}\right)&=\star d\left(|x|^2v(i(x))\s{dv(i(x))}{i(x)}\right)\\&=|x|^2v(i(x))\star d\left( \s{dv(i(x))}{i(x)}\right)+\s{dv(i(x))}{i(x)}\star d\left(|x|^2v(i(x))\right)
\end{align*}
so
\begin{align*}
	&-4\s{dv(i(x))}{i(x)}\star d\left(|x|^2v(i(x))\right)+2\star d\left(v(i(x))\s{dv(i(x))}{x}\right)\\
	&=2|x|^2v(i(x))\star d\left(\s{v(i(x))}{i(x)}\right)-2\s{dv(i(x))}{i(x)}\star d\left(|x|^2v(i(x))\right)
\end{align*}
Finally, we get
\begin{align}\label{0}
	&(\Delta_gw+2Kw)\star dw-\frac{1}{2}\star d|dw|_g^2=2v^2(i(x))\star d(|x|^2)+\frac{1}{|x|^2}\Delta v(i(x))\star d(|x|^2v(i(x)))\nonumber\\
	&+2|x|^2v(i(x))\star d\left(\s{dv(i(x))}{i(x)}\right)-2\s{dv(i(x))}{i(x)}\star d\left(|x|^2v(i(x))\right)+O\left(\frac{1}{|x|^2}\right)\nonumber\\
	&=(1)+(2)+(3)+(4)
\end{align}

We can develop $v$ at $0$ up to order two to get precise estimates of the boundary integral. We write for some real coefficients $\ens{a_{i,j}}_{i,j\geq 0}\subset \R$
\begin{align*}
	v(x)=\sum_{i,j=0}^ka_{i,j}x_1^ix_2^j+O(|x|^{k+1}).
\end{align*}
Now, we have
\begin{align*}
	\star d(|x|^2v(i(x))^2)=2v(i(x))^2(x_1dx_2-x_2dx_1)+2v(i(x))|x|^2\star dv(i(x)).
\end{align*}

as
\begin{align}\label{aij}
	a_{i,j}=\frac{\partial_1^i\partial_2^jv(0)}{i!j!}
\end{align}
We recall the three following formulas, valid for $x,y\in\R$;
\begin{align*}\label{prodcos}
2\cos x\cos x&=\cos(x+y)+\cos(x-y),\\
2\sin x\sin y&=\cos(x-y)-\cos(x+y),\\
2\cos x\sin y&=\sin(x+y)-\sin(x-y).
\end{align*}
Now
\begin{align*}
	\star dv(i(x))=\left(\frac{x_2^2-x_1^2}{|x|^4}\partial_1 v(i(x))-\frac{2x_1x_2}{|x|^4}\partial_2v(i(x))\right)dx_2-\left(\frac{x_1^2-x_2^2}{|x|^4}\partial_2 v(i(x))-\frac{2x_1x_2}{|x|^4}\partial_1v(i(x))\right)dx_1
\end{align*}
so using coordinates $x_1=r\cos(\theta)$, $x_2=r\sin(\theta)$, we get
\begin{align}
	\star dv(i(x))&=\frac{1}{r}\left(-\cos(2\theta)\p{1}v-\sin(2\theta)\p{2}v\right)\cos(\theta)+\frac{1}{r}\left(\cos(2\theta)\p{2}v-\sin(2\theta)\p{1}v\right)\sin(\theta)\nonumber\\
	&=-\left(\cos(2\theta)\cos(\theta)+\sin(2\theta)\sin(\theta)\right)\p{1}v+\left(\cos(2\theta)\sin(\theta)-\sin(2\theta)\cos(\theta)\right)\p{2}v\nonumber\\
	&=-\frac{1}{r}\left(\cos(\theta)\p{1}v+\sin(\theta)\p{2}v\right)
\end{align}

In particular, if $a,b\in\N$, with $a\neq b$, then
\begin{equation}\label{cosin}
	\int_{0}^{2\pi}\cos(ax)\sin(bx)dx=\int_{0}^{2\pi}\cos(ax)\cos(bx)dx=\int_{0}^{2\pi}\sin(ax)\sin(bx)dx=0
\end{equation}
So by \eqref{aij}, \eqref{cosin},
\begin{align*}
	&\int_{S^1_r}2v^2(i(x))\star (d|x|^2)=4\int_{S_r^1}v^2(i(x))(x_1dx_2-x_2dx_1)\\
	&=4\int_{0}^{2\pi}r^2\left(a_{0,0}+a_{1,0}\frac{\cos(\theta)}{r}+a_{0,1}\frac{\sin(\theta)}{r}+a_{2,0}\left(\frac{\cos(\theta)}{r}\right)^2+a_{0,2}\left(\frac{\sin(\theta)}{r}\right)^2+a_{1,1}\frac{\cos(\theta)\sin(\theta)}{r^2}+O(r^{-3})\right)^2d\theta\\
	&=4\int_{0}^{2\pi}\left(a_{0,0}^2r^2+2a_{0,0}\left(a_{2,0}\cos^2(\theta)+a_{0,2}\sin^2(\theta)\right)+a_{1,0}^2\cos^2(\theta)+a_{0,1}^2\sin^2(\theta)\right)d\theta\\
	&=8\pi a_{0,0}^2+8\pi a_{0,0}\left(a_{2,0}+a_{0,2}\right)+4\pi(a_{1,0}^2+a_{0,1}^2)\\
	&=8\pi r^2 v(p)^2+4\pi\left(v(p)\Delta v(p)+|dv(p)|^2\right)
\end{align*}
therefore
\begin{equation}\label{1}
	\int_{S_r^1}(1)=8\pi r^2 v(p)^2+4\pi\left(v(p)\Delta v(p)+|dv(p)|^2\right).
\end{equation}
Then,
\begin{align*}
&\int_{S_r^1}|x|^2\Delta(v(i(x)))\star d\left(|x|^2v(i(x))\right)=\int_{0}^{2\pi}2v(i(x))\Delta v(i(x))\frac{1}{|x|^2}\left(x_1dx_2-x_2dx_1\right)\\
&+\int_{0}^{2\pi}2v(i(x))\Delta v(i(x))\frac{1}{|x|^2}\bigg\{|x|^2\left(\frac{x_2^2-x_1^2}{|x|^4}\partial_1 v(i(x))-\frac{2x_1x_2}{|x|^4}\partial_2v(i(x))\right)dx_2\\
&-\left(\frac{x_1^2-x_2^2}{|x|^4}\partial_2 v(i(x))-\frac{2x_1x_2}{|x|^4}\partial_1v(i(x))\right)dx_1\bigg\}\\
&=4\pi v(p)\Delta v(p)+O(r^{-1}).
\end{align*}
thus
\begin{equation}\label{2}
	\int_{S_r^1}(2)=4\pi v(p)\Delta v(p)+O(r^{-1}).
\end{equation}

We now compute
\begin{align*}
	&|x|^2v(i(x))\star d\left(\s{dv(i(x))}{i(x)}\right)=|x|^2v(i(x))\bigg\{\d{1}\left(\frac{x_1}{|x|^2}\p{1}v+\frac{x_2}{|x|^2}\p{2}v\right)dx_2-\d{2}\left(\frac{x_1}{|x|^2}\p{1}v+\frac{x_2}{|x|^2}\p{2}v\right)dx_1\bigg\}\\
	&=|x|^2v(i(x))\bigg\{\left(\frac{x_2^2-x_1^2}{|x|^4}\p{1}v-\frac{2x_1x_2}{|x|^4}\p{2}v\right)dx_2-\left(\frac{x_1^2-x_2^2}{|x|^4}\p{2}v-\frac{2x_1x_2}{|x|^4}\p{1}v\right)\\
	&+\frac{x_1}{|x|^2}\left(\frac{x_2^2-x_1^2}{|x|^4}\p{1}^2v-\frac{2x_1x_2}{|x|^4}\p{12}^2v\right)dx_2+\frac{x_2}{|x|^2}\left(\frac{x_2^2-x_1^2}{|x|^4}\p{12}^2v-\frac{2x_1x_2}{|x|^4}\p{2}^2v\right)dx_2\\
	&-\frac{x_1}{|x|^2}\left(\frac{x_1^2-x_2^2}{|x|^4}\p{12}^2v-\frac{2x_1x_2}{|x|^4}\p{1}^2v\right)dx_1-\frac{x_2}{|x|^2}\left(\frac{x_1^2-x_2^2}{|x|^4}\p{2}^2v-\frac{2x_1x_2}{|x|^4}\p{12}^2v\right)dx_1\bigg\}
	\\
	&=\textnormal{(i)+(ii)}
\end{align*}
And we see that
\begin{align*}
	\mathrm{(i)}&=|x|^2v(i(x))\bigg\{\left(\frac{x_2^2-x_1^2}{|x|^4}\p{1}v-\frac{2x_1x_2}{|x|^4}\p{2}v\right)dx_2-\left(\frac{x_1^2-x_2^2}{|x|^4}\p{2}v-\frac{2x_1x_2}{|x|^4}\p{1}v\right)\bigg\}\\
	&=|x|^2v(i(x))\star dv(i(x))
\end{align*}
and we have already computed this integral, so 
\begin{equation}\label{3i}
	\int_{S_r^1}\mathrm{(i)}=-\pi\left(|dv(p)|^2+v(p)\Delta v(p)\right).
\end{equation}
and
\begin{align*}
	\int_{S_r^1}\mathrm{(ii)}&=\int_{S_r^1}v(i(x))\bigg\{\frac{1}{|x|^4}\left(x_1^2(x_2^2-x_1^2)-2x_1^2x_2^2\right)\p{1}^2v+\left(-2x_1^2x_2^2+x_2^2(x_1^2-x_2^2)\right)\p{2}^2v\bigg\}+O(r^{-1})\\
	&=-\int_0^{2\pi}v\left(\frac{\cos(\theta)}{r},\frac{\sin(\theta)}{r}\right)\left(\cos^2(\theta)\p{1}^2v\left(\frac{\cos(\theta)}{r},\frac{\sin(\theta)}{r}\right)+\sin^2(\theta)\p{2}^2v\left(\frac{\cos(\theta)}{r},\frac{\sin(\theta)}{r}\right)\right)d\theta+O(r^{-1})\\
	&=-\pi v(p)\Delta v(p)
\end{align*}
and finally
\begin{align}\label{3}
	\int_{S_r^1}(3)=2\int_{S_r^1}\textnormal{(i)+(ii)}=-2\pi\left(|dv(p)|^2+2v(p)\Delta v(p)\right)
\end{align}
We have only left $(4)$: 
\begin{align*}
	\int_{S_r^1}(4)&=-2\int_{S^1_r}\s{dv(i(x))}{i(x)}\star d\left(|x|^2v(i(x))\right)\\
	&=-2\int_{S_r^1}\s{dv(i(x))}{i(x)}v(i(x))2(x_1dx_2-x_2dx_1)\\
	&-2\int_{S^1_r}\s{dv(i(x))}{x}\star dv(i(x))\\
	&=-2\left\{\textnormal{(iii)+(iv)}\right\}
\end{align*}
Now, we easily compute
\begin{align*}
	\mathrm{(iii)}&=\int_{0}^{2\pi}\left(\p{1}v\frac{\cos(\theta)}{r}+\p{2}v\frac{\sin(\theta)}{r}\right)2vr^2d\theta\\
	&=2\int_{0}^{2\pi}\left(ra_{0,0}+a_{1,0}\cos(\theta)+a_{0,1}\sin(\theta)\right)\bigg\{\cos(\theta)\left(a_{1,0}+2a_{2,0}\frac{\cos(\theta)}{r}+2a_{1,1}\frac{\sin(\theta)}{r}\right)\\
	&+\sin(\theta)\left(a_{0,1}+2a_{0,2}\frac{\sin(\theta)}{r}+2a_{1,1}\frac{\cos(\theta)}{r}\right)\bigg\}d\theta+O(r^{-1})\\
	&=2\int_{0}^{2\pi}\left(a_{1,0}^2\cos^2(\theta)+a_{0,1}^2 \sin^2(\theta)+2a_{0,0}\left(a_{2,0}\cos^2(\theta)+a_{0,2}\sin^2(\theta)\right)\right)d\theta+O(r^{-1})\\
	&=2\pi(|dv(p)|^2+v(p)\Delta v(p))
\end{align*}
while
\begin{align*}
	\mathrm{(iv)}&=-\int_0^{2\pi}\left(r\cos(\theta)\p{1}v+r\sin(\theta)\p{2}v\right)\left(\frac{-1}{r}\left(\d{1}v\cos(\theta)+\d{2}v\sin(\theta)\right)\right)+O(r^{-1})\\
	&=-\int_{0}^{2\pi}(\p{1}v\,\cos(\theta)+\p{2}\sin(\theta))^2d\theta+O(r^{-1})\\
	&=-\pi|dv(p)|^2.
\end{align*}
Gathering estimates, we get
\begin{equation}\label{4}
	\int_{S_r^1}(4)=-2\pi\left(|dv(p)|^2+2v(p)\Delta v(p)\right)
\end{equation}
and we finally obtain by \eqref{0}, \eqref{1}, \eqref{2}, \eqref{3}, \eqref{3}, \eqref{4}
\begin{equation}\label{finald2k}
	\int_{S^1_r}\left((\Delta_gw+2Kw)\star dw-\frac{1}{2}\star d|dw|_g^2\right)=-8\pi r^2v^2(p)+O\left(\frac{1}{r}\right)
\end{equation}
and by \eqref{defresidu}, we get the correct multiplicative factor of $\alpha^2$ in front of this expression, which gives the correct expression as
\begin{align*}
	\mathrm{Res}_p\phi=2\alpha^2.
\end{align*}
This concludes the proof of the theorem.
\end{proof}

We shall verify that our is well-defined, as it seems at first glance singular. To do so, we need to estimate the integral 
\begin{align*}
	\frac{1}{2}\int_{D^2(p_j,R_0)\setminus D^2(p_j,R)}\left(\Delta_gw-2K_gw\right)^2d\vg,\quad 1\leq j\leq m
\end{align*}
 when $R\rightarrow 0$, where $w=|\phi|^2w^2$. Using the change of variable \eqref{changevariable}, this is equivalent to estimate
\begin{align*}
	\frac{1}{2}\int_{D^2(0,r)\setminus D^2(0,1)}\left(\Delta_gw-2K_gw\right)^2d\vg
\end{align*}
when $r\rightarrow\infty$, and $r^2=\dfrac{\alpha^2}{R^2}$. The asymptotic flatness coming from the proof of the theorem shows that
\begin{align*}
	&\frac{1}{2}\int_{D^2(0,r)\setminus D^2(0,1)}(\Delta_gw-2K_gw)^2d\vg\\
	&=\frac{1}{2}\int_{D^2(0,r)\setminus D^2(0,1)}\left(\Delta \left(|x|^2v\left(\frac{x}{|x|^2}\right)\right)-2K_g(x)|x|^2v\left(\frac{x}{|x|^2}\right)\right)^2d\leb^2(x)+O(1).
\end{align*}
Furthermore,
\begin{align*}
	\Delta \left(|x|^2v\left(\frac{x}{|x|^2}\right)\right)=4v\left(\frac{x}{|x|^2}\right)+\frac{4}{|x|^2}\left(x_1\partial_1v\left(\frac{x}{|x|^2}\right)+x_2\partial_{2}v\left(\frac{x}{|x|^2}\right)\right)+\frac{1}{|x|^2}\Delta v\left(\frac{x}{|x|^2}\right)=O(1)
\end{align*}
and recall that
\begin{align*}
	K_g(x)=O\left(\frac{1}{|x|^6}\right).
\end{align*}
so the only singular terms will come from $(\Delta_gw)^2$. Using polar coordinates $(x,y)=(\rho\cos(\theta),\rho\sin(\theta))$, $(\rho,\theta)\in [1,r]\times S^1$, 
\begin{align*}
	&\Delta \left(|x|^2v\left(\frac{x}{|x|^2}\right)\right)=4\left(a_{0,0}+a_{1,0}\frac{\cos(\theta)}{\rho}+a_{0,1}\frac{\sin(\theta)}{\rho}+a_{2,0}\frac{\cos^2(\theta)}{\rho^2}+\frac{\sin^2(\theta)}{\rho^2}+a_{1,1}\frac{\sin(2\theta)}{2\rho^2}\right)\\
	&-4\frac{\cos(\theta)}{\rho}\left(a_{1,0}+2a_{2,0}\frac{\cos(\theta)}{\rho}+a_{1,1}\frac{\sin(\theta)}{\rho}\right)-4\frac{\sin(\theta)}{\rho}\left(a_{0,1}+2a_{0,2}\frac{\sin(\theta)}{\rho}+a_{1,1}\frac{\cos(\theta)}{\rho}\right)\\
	&+\frac{2(a_{2,0}+a_{0,2})}{\rho^2}+O\left(\frac{1}{\rho^3}\right)\\
	&=4\left(a_{0,0}-a_{2,0}\frac{\cos^2(\theta)}{\rho^2}-a_{0,2}\frac{\sin^2(\theta)}{\rho^2}-a_{1,1}\frac{\sin(2\theta)}{2\rho^2}\right)+\frac{2(a_{2,0}+a_{0,2})}{\rho^2}+O\left(\frac{1}{\rho^3}\right)
\end{align*}
so
\begin{align*}
	&\frac{1}{2}\int_{D^2(0,r)\setminus D^2(0,1)}(\Delta_gw-2K_gw)^2d\vg\\
	&=\frac{1}{2}\int_{1}^r\int_{S^1}\left(16a_{0,0}^2-32a_{0,0}\left(a_{2,0}\frac{\cos^2(\theta)}{\rho^2}+a_{0,2}\frac{\sin^2(\theta)}{\rho^2}\right)+16\frac{a_{0,0}(a_{2,0}+a_{0,2})}{\rho^2}\right)\rho\, d\rho d\theta+O(1)\\
	&=16\pi a_{0,0}^2\int_{1}^{r}\rho^2d\rho-16a_{0,0}\left(a_{2,0}\left(\int_{1}^{r}\frac{d\rho}{\rho}\right)\left(\int_{S^1}\cos^2(\theta)d\theta\right)+a_{0,2}\left(\int_{1}^{r}\frac{d\rho}{\rho}\right)\left(\int_{S^1}\sin^2(\theta)d\theta\right)\right)\\
	&+16\pi a_{0,0}(a_{2,0}+a_{0,2})\int_{1}^{r}\frac{d\rho}{\rho}\\
	&=16\pi a_{0,0}^2\frac{(r^2-1)}{2}-16\pi a_{0,0}(a_{2,0}+a_{0,2})\log r+16\pi a_{0,0}(a_{2,0}+a_{0,2})\log r+O(1)\\
	&=8\pi a_{0,0}^2r^2+O(1)\\
	&=8\pi v^2(p_j)r^2+O(1)
\end{align*}
which proves comparing with \eqref{finald2k} that the expression \eqref{magique} is well-defined.

\subsection{Proof of main theorem \ref{main}}

\begin{proof}
	We are now in measure to prove the main theorem thanks of the preceding formula and Bryant's theorem \ref{theobryant}.
	Suppose first that $\Psi:S^2\rightarrow S^3$ is a non-umbilic Willmore sphere such that
	\begin{align*}
		m=\frac{1}{4\pi}\W(\Psi)
	\end{align*} 
	then it is the inversion at $0$ (after translation if necessary), after a stereographic projection in $\R^3$, of a minimal surface with $m$ embedded ends with zero logarithmic growth and finite total curvature. Furthermore, for all normal variation $\vec{v}=v\n$ of $\Psi$ such that $v\in W^{2,2}(S^2,\R)$, if $\phi=i\circ \Psi$, $w=|\phi|^2v$, we have by \eqref{magique},
    \[
			D^2\W(\psi)[\vec{v},\vec{v}]=\lim_{R\rightarrow 0}\frac{1}{2}\int_{\Sigma_R}(\Delta_g w-2K_gw)^2d\vg-
			4\pi\sum_{j=1}^m\frac{\mathrm{Res}_{p_j}(\phi)}{R^2}\,v^2(p_j)
	\]
	In particular, if $v(p_j)=0$ for all $1\leq j\leq m$, then
	\begin{align*}
		D^2 \W(\psi)[\vec{v},\vec{v}]\geq 0.
	\end{align*}
	If we introduce the map
	\begin{align*}
	&V : W^{2,2}(S^2,\R)\rightarrow \R^m\\
	&v\mapsto(v(p_1),\cdots, v(p_m))
	\end{align*}
	Then $V^{-1}(0)\subset W^{2,2}(S^2,\R)$ is a closed sub-space of codimension $m$. And the second variation of $W$ is non-negative on the complement of $V^{-1}(0)$, so the dimension of the subspace of $W^{2,2}(S^2,\R)$ where $D^2\W(\psi)$ is negative is bounded by $m$. Therefore, $\mathrm{Ind}_\W(\Psi)\leq m$.
	
	Now, we treat the umbilical case (where $m=1$). Thanks of the result of Bryant (\cite{bryant}), if there exists an umbilical point, then $\psi:S^2\rightarrow S^3$ is totally umbilical, and is a geodesic $2$-sphere. In particular $\psi:S^2\rightarrow S^3$ is a minimal immersion. As the Gauss map of a minimal surface is holomorphic, it is \textit{a fortiori} harmonic, so
	\begin{align*}
		\Delta_g \n+|d\n|_g^2\,\n=0
	\end{align*}
	if $g=\psi^*g_{S^3}$, and $\n:S^2\rightarrow S^2$ is the Gauss map of $\psi$. The Jacobi operator of the minimal surface $\psi:S^2\rightarrow S^3$ is simply
	\begin{align*}
		L_g=\Delta_g+(|\vec{\I}_g|^2+2)=\Delta_g+(|d\n|_g^2+2)
	\end{align*} 
	so for all $a\in\R^4$, we have
	\begin{align*}
		L_g(a\cdot\n)=a\cdot\left(\Delta_g\n+|d\n|_g^2\,\n\right)+2a\cdot\n=2a\cdot \n
	\end{align*}
	so $2$ is a positive eigenvalue of the Jacobi operator, and one can show that it is the only one.. Furthermore, the associated eigenspace is $1$-dimensional (see the paper of Frederick Almgren \cite{almgren}) so $\psi:S^2\rightarrow S^3$ is of index $1$.
	 Therefore, by \eqref{indicewillmin} and \eqref{formuleindicemin} $\psi$ is stable \textit{i.e.} $\mathrm{Ind}_{\W}(\psi)=0$, and the bound is trivially verified.
\end{proof}

\begin{rem}
	This proof shows in particular that whenever in the classification of Bryant, we have an Willmore immersion $\psi:\Sigma\rightarrow S^3$ from a closed Riemann surface $\Sigma$ which is not totally umbilic, with Bryant's quadratic form $\mathscr{Q}_{\psi}$ identically equal to $0$ (theorem E in \cite{bryant}), then for some stereographic projection $\pi:S^3\rightarrow \R^3$, $\pi\circ \psi:\Sigma\rightarrow\R^3$ is a branched minimal immersion with embedded planar ends. In particular $\W(\psi)$ is quantized by $4\pi$ and
	\begin{align*}
		\mathrm{Ind}_{\W}(\psi)\leq \frac{1}{4\pi}\W_{S^3}(\psi).
	\end{align*}  
    However, when the genus of $\Sigma$ is larger than $1$, then there are Willmore immersions with non-zero quadratic form $\mathscr{Q}_{\psi}$. One example is furnished by the Clifford torus in $S^3$ with energy $2\pi^2$, which cannot be the stereographic projection of a minimal surface in $\R^3$.
\end{rem}

\subsection{Index and Schrödinger operators}\label{quadraticschro}

Let $\bar{\Sigma}$ a closed Riemann surface and $\Sigma=\bar{\Sigma}\setminus\ens{p_1,\cdots, p_m}$ a Riemann surface with $m$ punctured points $\ens{p_1,\cdots, p_m}\subset \bar{\Sigma}$. If $\phi:\Sigma\rightarrow\R^3$ is a minimal immersion, recall that for a normal variation $\w=w\n$, we have
\begin{align*}
	D^2A(\phi)[\w,\w]&=\int_{\Sigma}\left(|dw|_g^2+2K_gw^2\right)d\vg\\
	&=-\int_{\Sigma}^{}w\left(L_gw\right)d\vg
\end{align*}
if $L_g=\Delta_g-2K_g=\Delta_g+|dN|_g^2$ if $N:\Sigma\rightarrow S^2$ is the Gauss map of $\phi$. If $\phi$ conformal, and if $d\mathrm{vol}_{\Sigma}$ is a canonical volume form on $\Sigma$, we have,
\begin{align*}
    d\mathrm{vol}_g&=\frac{|d\phi|^2}{2}d\mathrm{vol}_\Sigma\\
	-2K_g&=|dN|_g^2=2|d\phi|^{-2}|dN|^2
\end{align*}
so if $M=\phi(\Sigma)$ is a finite curvature minimal surface with $m$ embedded planar ends, and $\psi:\Sigma\rightarrow \R^3$ is the inversion at $0$ of $\phi$, we have if $\vec{v}=v\n$ is a normal variation, by \eqref{magique}
\begin{align*}
	D^2\W(\psi)[\vec{v},\vec{w}]&=\lim\limits_{R\rightarrow 0}\frac{1}{2}\int_{\Sigma_R}^{}\left(L_g(|\phi|^2v)\right)^2d\vg-4\pi\sum_{j=1}^{m}\frac{\mathrm{Res}_{p_j}\phi}{R^2}v^2(p_j)\\
	&=\lim\limits_{R\rightarrow 0}\frac{1}{2}\int_{\Sigma_R}^{}\left(2|d\phi|^{-2}\Delta_{\Sigma}(|\phi|^2v)+2|d\phi|^{-2}|dN|^2|\phi|^2v^2\right)^2\frac{|d\phi|^2}{2}d\mathrm{vol}_{\Sigma}-4\pi\sum_{j=1}^{m}\frac{\mathrm{Res}_{p_j}\phi}{R^2}v^2(p_j)\\
	&=\lim_{R\rightarrow 0}\int_{\Sigma_R}^{}\left(L_{\phi}(|\phi|^2v)\right)^2\frac{d\mathrm{vol}_{\Sigma}}{|d\phi|^{2}}-4\pi\sum_{j=1}^{m}\frac{\mathrm{Res}_{p_j}\phi}{R^2}v^2(p_j)
\end{align*}
if $L_{\phi}=\Delta_{\Sigma}+|dN|^2$. Moreover, the Gauss map $N:\Sigma\rightarrow S^2$ is a holomorphic map, so it is in particular harmonic, \textit{i.e.}
\begin{align*}
	\Delta_{\Sigma}N+|dN|^2N=0.
\end{align*}
Therefore (recalling \eqref{defresidu} for the definition of the residue) we can study in general the problem of finding the index of the following quadratic form
\begin{align*}
	Q_{f}(v,v)=\lim\limits_{R\rightarrow 0}\int_{\Sigma_R}^{}\left(L_f(|\Re f|^2v)\right)^2\frac{d\mathrm{vol}_{\Sigma}}{|\Re df|^2}-4\pi\sum_{j=1}^m\frac{|\mathrm{Res}_{p_j}f(z)dz|^2}{R^2 }v^2(p_j)
\end{align*}
where $f:\Sigma\rightarrow \C^3$ is a meromorphic immersion with at most simple poles at each end $p_j\in\bar{\Sigma}$ ($1\leq j\leq m$), and $N:\Sigma\rightarrow S^2$ is the holomorphic Gauss map of $\Re f:\Sigma\rightarrow \R^3$. 

\begin{rem}
	We remark that for all conformal transformation of $\varphi:\bar{\Sigma}\rightarrow\bar{\Sigma}$, we have
	\begin{align*}
		Q_{f\circ\varphi}(v\circ \varphi,v\circ\varphi)=Q_f(v,v)
	\end{align*}
	This is easily seen by the conformal invariance of the Laplacian in two-dimensions (see the book of Frédéric Hélein \cite{helein}). 
\end{rem}

\nocite{}
\bibliographystyle{alpha}
\bibliography{biblio}

\begin{thebibliography}{{Alm}66}

\bibitem[{Alm}66]{almgren}
Frederick~J. {Almgren Jr.}
\newblock Some {I}nterior {R}egularity {T}heorems for {M}inimal {S}urfaces and
  an {E}xtension of {B}ernstein's {T}heorem.
\newblock {\em Annals of Mathematics 84}, 1966.

\bibitem[Bla29]{blaschke}
Wilhelm J.~E. Blaschke.
\newblock {\em Vorlesungen \"{U}ber {D}ifferentialgeometrie {III}:
  {D}ifferentialgeometrie der {K}reise und {K}ugeln}.
\newblock Springer-{V}erlag, coll. {G}rundlehren der mathematischen
  {W}issenschaften, 1929.

\bibitem[Bry84]{bryant}
Robert~L. Bryant.
\newblock A duality theorem for {W}illmore surfaces.
\newblock {\em Journal of {D}ifferential {G}eometry, 20, 23-53}, 1984.

\bibitem[Che74]{chen1}
Bang-Yen Chen.
\newblock Some conformal invariants of submanifolds and their applications.
\newblock {\em Boll. Un. Mat. Ital. 10, no. 4, 380–385}, 1974.

\bibitem[Che15]{chen2}
Bang-Yen Chen.
\newblock {\em Total {M}ean {C}urvature and {S}ubmanifolds of {F}inite {T}ype,
  \textit{Second {E}dition}}.
\newblock Series in {P}ure {M}athematics, 2015.

\bibitem[CM11]{coldingminicozzi1}
Tobias~H. Colding and William~P. {Minicozzi II}.
\newblock {\em A {C}ourse in {M}inimal {S}urfaces}.
\newblock American Mathematical Society, Volume 121, 2011.

\bibitem[CM14]{chodosh}
Otis Chodosh and Davi Maximo.
\newblock The {T}opology and {I}ndex of {M}inimal {S}urfaces.
\newblock {\em \href{http://arxiv.org/abs/1405.7356}{Preprint}}, 2014.

\bibitem[CT94]{tysk2}
Shiu-Yuen Cheng and Johan Tysk.
\newblock Schr{\"o}dinger {O}perators and {I}ndex {B}ounds for {M}inimal
  {S}ubmanifolds.
\newblock {\em Rocky {M}ountain {J}ournal of {M}athematics, Vol. 24, No. 3,
  Summer}, 1994.

\bibitem[EK93]{ejiri1}
Norio Ejiri and Motoko Kotani.
\newblock Index and {F}lats {E}nds of {M}inimal {S}urfaces.
\newblock {\em Tokyo J. Math, Vol. 16, No. 1}, 1993.

\bibitem[FC85]{fischer}
Doris Fischer-Colbrie.
\newblock On complete minimal surfaces with finite {M}orse index in three
  manifolds.
\newblock {\em Inventiones mathematicae, 82, 121-132}, 1985.

\bibitem[Fed69]{federer}
Herbert Federer.
\newblock {\em Geometric {M}easure {T}heory}.
\newblock Springer-Verlag, 1969.

\bibitem[Hé96]{helein}
F.~Hélein.
\newblock {\em Applications harmoniques, lois de conservation, et repères
  mobiles}.
\newblock Diderot éditeur, Sciences et Arts, 1996.

\bibitem[Hub57]{huber}
Alfred Huber.
\newblock On subharmonic functions and differential geometry in the large.
\newblock {\em Comment. Math. Helv., Vol. 32, 13-72.}, 1957.

\bibitem[LM99]{lopezmartin}
Francisco~J. L{\'o}pez and Francisco Mart\'{i}n.
\newblock {\em Complete minimal surfaces in $\R^3$}.
\newblock Publicacions Matem\`{a}tiques, Vol. 43, No. 2, 341-449, 1999.

\bibitem[L{\'o}p92]{lopez}
Francisco~J. L{\'o}pez.
\newblock The classification of complete minimal surfaces with total curvature
  greate than $-12\pi$.
\newblock {\em Transactions of the American Mathematical Society, Vol. 334, No.
  1, November}, 1992.

\bibitem[LY82]{lieyau}
Peter Li and Shing-Tung Yau.
\newblock A {N}ew {C}onformal {I}nvariant and {I}ts {A}pplications to the
  {W}illmore {C}onjecture and the {F}irst {E}igenvalue of {C}ompact {S}urfaces.
\newblock {\em Inventiones mathematicae 69, 269-291}, 1982.

\bibitem[MR06]{montielRos}
Sebasti\'{a}n Montiel and Antonio Ros.
\newblock {\em Schr\"{o}dinger operators associated to a holomorphic map}.
\newblock Global Differential Geometry and Global Analysis, Volume 1481 of the
  series Lecture Notes in Mathematics, 147-174, 2006.

\bibitem[Oss63]{osserman}
Robert Osserman.
\newblock On {C}omplete {M}inimal {S}urfaces.
\newblock {\em Archive for Rational Mechanics and Analysis, December, Volume
  13, Issue 1, pp 392-404}, 1963.

\bibitem[Pau14]{paulin}
Fr\'{e}d\'{e}ric Paulin.
\newblock {\em Groupes et g\'{e}om\'{e}tries}.
\newblock Notes de cours, 2014.

\bibitem[Riv08]{riviere1}
Tristan Rivi\`{e}re.
\newblock Analysis aspects of {W}illmore surfaces.
\newblock {\em Inventiones mathematicae, 174, 1-45}, 2008.

\bibitem[Riv13]{rivierepcmi}
Tristan Rivi\`{e}re.
\newblock {\em Weak immersions of surfaces with ${L}^2$-bounded second
  fundamental form}.
\newblock PCMI Graduate Summer School, 2013.

\bibitem[Riv14]{rivierecrelle}
Tristan Rivière.
\newblock Variational principles for immersed surfaces with ${L}^2$-bounded
  second fundamental form.
\newblock {\em Journal f\"{u}r die reine und angewandte {M}athematik 695,
  41–98}, 2014.

\bibitem[Riv15]{coutsphere}
Tristan Rivi\`{e}re.
\newblock Willmore {M}inimax {S}urfaces and the {C}ost of the {S}phere
  {E}version.
\newblock {\em Preprint}, 2015.

\bibitem[Sch83]{schoenPlanar}
Richard~M. Schoen.
\newblock {U}niqueness, {S}ymmetry, and {E}mbeddedness of {M}inimal {S}urfaces.
\newblock {\em Journal of {D}ifferential {G}eometry, 18, 791-809}, 1983.

\bibitem[Tho23]{thomsen}
Gerhard Thomsen.
\newblock \"{U}ber konforme {G}eometrie, {I}: {G}rundlagen der konformen
  {F}l\"{a}chentheorie.
\newblock {\em Hamb. Math. Abh. 3}, 1923.

\bibitem[Tys87]{tysk}
Johan Tysk.
\newblock Eigenvalue estimates with applications to minimal surfaces.
\newblock {\em Pacific Journal of Mathematics, Vol. 128, No. 2}, 1987.

\bibitem[Wei78]{weiner}
Joel~L. Weiner.
\newblock On a {P}roblem of {C}hen, {W}illmore, \textrm{et al}.
\newblock {\em Indiana University Mathematics Journal, Vol. 27, No. 1}, 1978.

\bibitem[Wil82]{willmore2}
Thomas~J. Willmore.
\newblock {\em Total curvature in {R}iemannian geometry}.
\newblock Ellis Horwood Series Mathematics and its applications, 1982.

\end{thebibliography}

\end{document}